\documentclass[11pt]{amsart}
\usepackage[totalwidth=440pt, totalheight=640pt]{geometry}
\usepackage{amsmath, amssymb, amsthm, amsfonts, mathrsfs, eucal, enumerate, natbib, layout}
\usepackage[normalem]{ulem}
\usepackage[all,tips,2cell]{xy}
\usepackage[usenames]{color}
\usepackage{verbatim}
\usepackage{mathtools}

\newcommand{\details}[1]{}

\newtheorem{theorem}{Theorem}[section]
\newtheorem*{theorem*}{Theorem}
\newtheorem{corollary}[theorem]{Corollary}
\newtheorem*{corollary*}{Corollary}
\newtheorem{lemma}[theorem]{Lemma}

\newtheorem*{claim*}{Claim}
\newtheorem*{lemma*}{Lemma}
\newtheorem{proposition}[theorem]{Proposition}
\newtheorem*{proposition*}{Proposition}

\newtheorem*{conjecture*}{Conjecture}
\newtheorem{def-proposition}[theorem]{Definition-Proposition}
\theoremstyle{definition}
\newtheorem{definition}[theorem]{Definition}
\newtheorem*{definition*}{Definition}
\newtheorem{remark}[theorem]{Remark}

\newtheorem*{example*}{Example}
\numberwithin{equation}{section}

\newcommand{\bS}{\mathbf{S}}

\newcommand{\ZZ}{\mathbb{Z}}

\newcommand{\GG}{\mathbb{G}}

\newcommand{\GL}{\mathrm{GL}}
\newcommand{\PGL}{\mathrm{PGL}}
\newcommand{\End}{\mathrm{End}}

\newcommand{\uAut}{\underline{\mathrm{Aut}}}
\newcommand{\uEnd}{\underline{\mathrm{End}}}
\newcommand{\uIsom}{\underline{\mathrm{Isom}}}

\newcommand{\cTors}{\mathcal{T}\mathrm{ors}}
\newcommand{\cTorsRig}{\mathcal{T}\mathrm{orsRig}}
\newcommand{\cCub}{\mathcal{C}\mathrm{ub}}

\newcommand{\bTors}{\mathbb{T}\mathrm{ors}}
\newcommand{\tors}{\mathrm{Tors}}
\newcommand{\bGerbe}{\mathbb{G}\mathrm{erbe}}

\newcommand{\bbPicard}{\mathbb{P}\mathrm{icard}}

\newcommand{\HH}{\mathrm{H}}
\newcommand{\R}{\mathrm{R}}

\newcommand{\Br}{\mathrm{Br}}
\newcommand{\Az}{\mathrm{Az}}
\newcommand{\Pic}{\mathrm{Pic}}

\newcommand{\Lf}{\mathrm{Lf}}

\newcommand{\Ob}{\mathrm{Ob}}
\newcommand{\Ab}{\mathrm{Ab}}
\newcommand{\res}{\mathrm{res}}

\newcommand{\id}{\mathrm{id}}
\newcommand{\st}{\mathrm{st}}

\newcommand{\cH}{\mathcal{H}}
\newcommand{\cB}{\mathcal{B}}

\newcommand{\cG}{\mathcal{G}}
\newcommand{\cA}{\mathcal{A}}
\newcommand{\cAut}{\mathcal{A}\mathrm{ut}}
\newcommand{\cD}{\mathcal{D}}

\newcommand{\cL}{\mathcal{L}}
\newcommand{\cM}{\mathcal{M}}
\newcommand{\cP}{\mathcal{P}}

\newcommand{\cO}{\mathcal{O}}
\newcommand{\cX}{\mathcal{X}}
\newcommand{\cF}{\mathcal{F}}
\newcommand{\cE}{\mathcal{E}}
\newcommand{\cY}{\mathcal{Y}}

\newcommand{\ass}{\mathsf{a}}
\newcommand{\comm}{\mathsf{c}}

\newcommand{\twoP}{\mathbb{P}}

\newcommand{\twoX}{\mathbb{X}}

\begin{document}

\title[Brauer groups of 1-motives]
{Brauer groups of 1-motives}

\author{Cristiana Bertolin}
\address{Dipartimento di Matematica, Universit\`a di Torino, Via Carlo Alberto 10, Italy}
\email{cristiana.bertolin@unito.it}
\author{Federica Galluzzi}
\address{Dipartimento di Matematica, Universit\`a di Torino, Via Carlo Alberto 10, Italy}
\email{federica.galluzzi@unito.it}






\begin{abstract} 
Over a normal base scheme, we prove the generalized Theorem of the Cube for 1-motives and that a torsion class of the group $\HH^2_{\acute {e}t}(M,\GG_{m,M})$ of a 1-motive $M$, whose pull-back via the unit section 
$\epsilon : S \rightarrow M$ is zero, comes from an Azumaya algebra. In particular, we deduce that over an algebraically closed field of characteristic zero, all classes of $\HH^2_{\acute {e}t}(M,\GG_{m,M})$ come from Azumaya algebras. 

\end{abstract}


\maketitle


\tableofcontents

\section*{Introduction}

 Grothendieck has defined the Brauer group $\Br(X)$ of a scheme $X$ as the group of similarity classes of Azumaya algebras over $X$. In \cite[I, \S 1]{Grothendieck68} he constructed an injective group homomorphism  
    \begin{equation} \label{eq:DeltaGrothendieck}
\delta: \Br(X) \longrightarrow \HH^2_{\acute {e}t}(X,\GG_m)
\end{equation}

 \noindent from the Brauer group of $X$ to the \'etale cohomology group $\HH^2_{\acute {e}t}(X,\GG_m)$. 
This homomorphism is not in general bijective, as pointed out by Grothendieck in 
\cite[II, \S 2]{Grothendieck68}, where he found a scheme $X$ whose Brauer group is a torsion group but whose \'etale cohomology group $\HH^2_{\acute {e}t}(X,\GG_m)$ is not torsion.
However, if $X$ is quasi-compact the elements of $\delta (\Br(X))$ are torsion elements of $\HH^2_{\acute {e}t}(X,\GG_m),$ and so Grothendieck asked in loc. cit. the following question:

QUESTION: \emph{For a quasi-compact scheme $X$, is the image of $\Br(X)$ via the homomorphism $\delta$ (\ref{eq:DeltaGrothendieck}) the torsion subgroup $\HH^2_{\acute {e}t}(X,\GG_m)_\tors$ of $\HH^2_{\acute {e}t}(X,\GG_m)$?}

Grothendieck showed that if $X$ is regular, the \'etale cohomology group $\HH^2_{\acute {e}t}(X,\GG_m)$ is a torsion group, and so under this hypothesis the question is whether the Brauer group of $X$ is all of $ \HH^2_{\acute {e}t}(X, \GG_m).$

The following well-known results are related to this question:
If $X$ has dimension $\leq 1$ or if $X$ is regular and of dimension $\leq 2,$ then the Brauer group of $X$ is all of  $\HH^2_{\acute {e}t}(X, \GG_{m,X})$ (\cite[II, Cor 2.2]{Grothendieck68}). Gabber (unpublished theorem) showed that the Brauer group of a quasi-compact and separated scheme $X$ endowed with an ample invertible sheaf is isomorphic to $ \HH^2_{\acute {e}t}(X, \GG_m)_{\tors}$. A different proof of this result was found by de Jong (see \cite{deJong}).

In \cite{Giraud} Giraud introduced gerbes in the general setting of non abelian cohomology following Grothendieck's ideas: in particular he proved that gerbes give a geometrical description of classes of the group $ \HH^2(X, \GG_m)$. 

The aim of this paper is to extend Grothendieck's theory of Brauer groups to 1-motives, using gerbes as fundamental tools.

In particular, 
\begin{itemize}
   \item we study gerbes on stacks which are not separated;
   \item we study Azumaya algebras and Brauer groups for stacks which are not separated;
   \item we apply the above results to 1-motives using the dictionary between
length two complexes of abelian sheaves and Picard stacks
 developed by Deligne in \cite[Expos\'e XVIII, \S 1.4]{SGA4D}. Remark that the Picard stacks associated to 1-motives are not algebraic in the sense of \cite{LaumonMoretB} since they are not quasi-separated. 
\end{itemize}

We proceed in the following way:

Let $\bS$ be a site. In Section \ref{recall} we associate to a stack in groupoids $\cX$ over $\bS$ the site $\bS(\cX)$, which allows us to study the notion of sheaf and gerbe on a stack. 

In Section \ref{bands} we prove the following homological interpretation of $F$-gerbes, with $F$ an abelian sheaf 
on a site $\bS$:  
 the Picard 2-stack $\bGerbe_{\bS} (F)$ of $F$-gerbes is equivalent (as Picard 2-stack) to the Picard 2-stack associated to the complex $F[2]$, where $F[2]=[F \to 0 \to 0]$ with $F$ in degree -2:
\begin{equation}\label{eq:homointergerbes}
\bGerbe_{\bS} (F) \cong 2\st (F[2])
\end{equation}
 (Theorem \ref{thm:H^2-general}). In particular, for $i=2,1,0,$ we have an isomorphism of abelian groups between the $i$-th classifying group $ \bGerbe^i_{\bS}(F)$ and the cohomological group $\HH^i(\bS,F).$ The equivalence of Picard 2-stacks (\ref{eq:homointergerbes}) contains the following classical result: elements of $ \bGerbe^2_{\bS}(F)$, which are $F$-equivalence classes of $F$-gerbes, are parametrized by cohomological classes of
 $\HH^2(\bS,F)$. Always in Section \ref{bands},
applying \cite[Chp IV]{Giraud} to the site $\bS(\cX)$ associated to a stack $\cX$, 
 we obtain the Picard 2-stack
 $\bGerbe_{\bS(\cX)} (\cF)$ of $\cF$-gerbes on $\cX$, with $\cF$ an abelian sheaf on the site $\bS(\cX)$. 
We finish Section \ref{bands} proving the effectiveness of the 2-descent of $\GG_m$-gerbes with respect to a faithfully flat morphism of schemes which is quasi-compact or locally of finite presentation (Theorem \ref{thm:DescentGerbesOverSchemes}).

Let $\bS_{\acute {e}t}$ be the \'etale site on an arbitrary scheme $S$ and let $\cX=(\cX, \cO_\cX)$ be a locally ringed $S$-stack with associated \'etale site $\bS_{\acute {e}t} (\cX).$ In Section \ref{brauergroup} we recall the notion of the Brauer group $\Br(\cX)$ of $\cX$ and in Theorem \ref{thm:BR(X)inj} we establish an injective group homomorphism  
\begin{equation}\label{introDelta}
\delta: \Br(\cX) \longrightarrow  \HH^2_{\acute {e}t}(\cX,\GG_{m, {\cX}}).
\end{equation}
which extends Grothendieck's group homomorphism (\ref{eq:DeltaGrothendieck}) to locally ringed $S$-stacks.

Let $M=[u:X \rightarrow G]$ be a 1-motive defined over a scheme $S$,
 with $X$ an $S$-group scheme which is, locally for the \'etale
topology, a constant group scheme defined by a finitely generated free
$\ZZ \,$-module, $G$ an extension of an abelian $S$-scheme by an $S$-torus, and finally 
 $u:X \rightarrow G $ a morphism of $S$-group schemes.
Since in \cite[Expos\'e XVIII, \S 1.4]{SGA4D} Deligne associates to any length two complex of abelian sheaves a Picard stack, in Section \ref{geometricalcase} we can define the Brauer group of the 1-motive $M$ as the Brauer group $ \Br(\cM)$ of the associated Picard stack $\cM$ and by Theorem \ref{thm:BR(X)inj} we have an injective group homomorphism $\delta: \Br(\cM) \rightarrow  \HH^2_{\acute {e}t}(\cM,\GG_{m, {\cM}}).$
At the end of Section \ref{geometricalcase} we prove the effectiveness of the descent of Azumaya algebras and of $\GG_m$-gerbes with respect to the quotient map $\iota : G \to [G/X] \cong \cM $ (Lemma \ref{lem:DescentAzumaya} and Lemma \ref{lem:DescentGerbes}).

Denote by $s_{ij}:=   \cM \times_{\bS} \cM  \to \cM \times_{\bS} \cM \times_{\bS} \cM  $ the map
which inserts the unit section $\epsilon: {\bS} \rightarrow \cM$ of $\cM$ into the $k$-th factor for $k \in \{1,2,3 \}-\{i,j\} .$
If $\ell$ is a prime number distinct from the residue characteristics of $S$,
we say that the 1-motive $M$ satisfies \textbf{the generalized Theorem of the Cube} for the prime $\ell$ if the homomorphism 
\[
\begin{matrix}
\prod_{(i,j) \in \{1,2,3 \}} s_{ij}^{*} : \HH_{\acute {e}t}^2(\cM^3,\GG_{m, \cM^3})(\ell) & \longrightarrow &  \big( \HH_{\acute {e}t}^2(\cM^2 ,\GG_{m, \cM^2})(\ell) \big)^3\\
x &  \longmapsto & (s_{12}^{*} (x), s_{13}^{*} (x),s_{23}^{*} (x))
\end{matrix}
\]
 is injective, where $(\ell)$ denotes the $\ell$-primary component (Definition \ref{def-thmCube}). We start Section \ref{proofGTC} studying the consequences of the generalized Theorem of the Cube for 1-motives. 
In Corollary \ref{cor:HythForThmCubeOnG} we show that if the base scheme is connected, reduced, normal and noetherian, extensions of abelian schemes by split tori satisfy the generalized Theorem of the Cube for any prime $\ell$ distinct from the residue characteristics of $S$ (Corollary \ref{cor:HythForThmCubeOnG}). Then, as a consequence of the effectiveness of the 2-descent of $\GG_m$-gerbes with respect to the quotient map $\iota : G \to [G/X] \cong \cM $ (Lemma \ref{lem:DescentGerbes}), we get Theorem \ref{thm:HythForThmCubeOnM}: 
1-motives, which are defined over a connected, reduced, normal and noetherian scheme $S$, and whose underlying tori are split, satisfy the generalized Theorem of the Cube for any prime $\ell$ distinct from the residue characteristics of $S$. Note that
in \cite[Thm 5.1]{BB} S. Brochard and the first author prove the Theorem of the Cube (involving the $\HH^1(\cM,\GG_{m,\cM})$ instead of $\HH^2(\cM,\GG_{m,\cM})$) for 1-motives, and in \cite{BF} the authors show that the sheaf of divisorial correspondences of extensions of abelian schemes by tori is representable.

In Section \ref{proofmainTHM} we investigate Grothendieck's QUESTION for 1-motives and
our answer is contained in Theorem \ref{mainTHM} which states that if    
$M=[u:X \rightarrow G]$ is 1-motive defined over a normal and noetherian scheme $S$ and if the extension $G$ underlying $M$ satisfies the generalized Theorem of the Cube for a prime number $\ell$ distinct from the residue characteristics of $S$, then the $\ell$-primary component of the kernel of the homomorphism $\HH^2_{\acute {e}t}(\epsilon):\HH^2_{\acute {e}t}(\cM,\GG_{m,\cM})\rightarrow \HH^2_{\acute {e}t}(S,\GG_{m,S})$ induced by the unit section 
$\epsilon : S \rightarrow M$ of $M,$ is contained in the Brauer group of $M$:
$$\ker \big[\HH^2_{\acute {e}t}(\epsilon):\HH^2_{\acute {e}t}(\cM,\GG_{m,\cM})\longrightarrow \HH^2_{\acute {e}t}(S,\GG_{m,S})\big] (\ell) \subseteq \Br(\cM).$$ 
 
\par\noindent We prove this result as follows: first we show this theorem for 
an extension of an abelian scheme by a torus using Hoobler's Theorem \cite[Thm 3.3]{Hoobler72}  
(Proposition \ref{mainThmG}). Then, thanks to the effectiveness of the descent of Azumaya algebras and of $\GG_m$-gerbes with respect to the quotient map $\iota : G \to [G/X] \cong \cM ,$ we get the required statement for $M$. We finish Section \ref{proofmainTHM} giving a positive answer to Grothendieck's QUESTION for 1-motives (and so in particular for semi-abelian varieties) over an algebraically closed field of characteristic zero (Corollary \ref{Br(M)}).

In the last years, several authors have worked with the Brauer group of stacks (see for example \cite{AM16}, \cite{EHKV01}, \cite{Lieblich08}) but most of them focus on algebraic or separated stacks. Moreover the techniques used in this paper are rather different from the ones used in \cite{AM16}, \cite{EHKV01}, \cite{Lieblich08}. Since the Picard stack associated to a 1-motive is not quasi-separated, we recall the theory of Brauer group of stacks.

An important role in this paper is played by the 2-descent theory of gerbes for which we add an Appendix. 

\section*{Acknowledgment}
We are very grateful to Pierre Deligne for his comments on the first version of this paper and for his communication on 2-descent theory for stacks (see Appendix). We would like to thank also the referee for the very useful comments. 

\section*{Notation}

\textbf{Stack language}\\
Here we refer mainly to \cite{Giraud}. Let $\bS$ be a site. A \textbf{stack} over $\bS$ is a fibered category $\cX$ over $\bS$ such that 
\begin{itemize}
   \item (\emph{Gluing condition on objects}) descent is effective for objects in $\cX$, and
   \item (\emph{Gluing condition on arrows}) for any object $U$ of $\bS$ and for every pair of objects $X,Y$ of the category $\cX(U)$, the presheaf of arrows $\mathrm{Arr}_{\cX(U)}(X,Y)$ of $\cX(U)$ is a sheaf over $U$.
\end{itemize}
 For the notions of morphismsm of stacks (i.e. catesian functors) and morphism of cartesian functors we refer to \cite[Chp. II 1.2]{Giraud}. 
 An \textbf{equivalence} (resp. \textbf{isomorphism}) \textbf{of stacks} $F: \cX \to \cY$ is a
  morphism of stacks which is an equivalence  (resp. isomorphism) of fibered categories over $\bS$, that is 
$F(U): \cX(U) \to \cY(U)$ is an equivalence (resp. isomorphism) of categories for any object $U$ of $\bS$.
A \textbf{stack in groupoids} over $\bS$ is a stack $\cX$ over $\bS$ such that for any object $U$ of $\bS$ the category $\cX(U)$ is a groupoid, i.e. a category in which all arrows are invertible. Recall that 2-morphisms of stacks in groupoids are automatically invertible. \emph{From now on, all stacks will be stacks in groupoids.}

A \textbf{gerbe} over the site $\bS$ is a stack $\cG$ over $\bS$ such that
\begin{itemize}
	\item $\cG$ is locally not empty: for any object $U$ of $\bS$, there exists a covering $\{\phi_i: U_i \to U\}_{i \in I} $ for which the set of objects of the category $\cG (U_i)$ is not empty for all $i \in I$;
	 \item $\cG$ is locally connected: for any object $U$ of $\bS$ and for each pair of objects $g_1$ and $g_2$  of $\cG(U)$, there exists a covering $\{\phi_i: U_i \to U\}_{i \in I} $ of $U$ such that the set of arrows from $g_{1\vert U_i}$ to $g_{2\vert U_i}$ in $\cG(U_i)$ is not empty for all $i \in I$.
\end{itemize}
A \textbf{morphism} (resp. \textbf{isomorphism}) \textbf{of gerbes} is just a morphism (resp. isomorphism) of stacks whose source and target are gerbes, and a 2-morphism of gerbes is a morphism of cartesian functors. An \textbf{equivalence of gerbes} is an equivalence of stacks. 

A \textbf{strictly commutative Picard stack} over the site $\bS$ (just called a Picard stack) is a stack $\cP$ over $\bS$ endowed with a morphism of stacks $ \otimes: \cP \times_{\bS} \cP \rightarrow \cP,$ called the group law of $\cP$, and two natural isomorphisms $\ass$ and $\comm$, expressing the associativity and the commutativity constraints of the group law of $\cP$,
such that $\cP (U)$ is a strictly commutative Picard category for any object $U$ of $\bS$ (see \cite{SGA4D} 1.4.2 for more details).
 An \textbf{additive functor} $(F,\sum):\cP_1 \rightarrow \cP_2 $
between two Picard stacks is a morphism of stacks $F: \cP_1 
\rightarrow \cP_2$ endowed with a natural isomorphism $\sum: F(a\otimes_{\cP_1} b) \cong 
F(a) \otimes_{\cP_2} F(b)$ (for all $a,b \in \cP_1$) which is compatible with the natural 
isomorphisms $\ass$ and $\comm$ underlying 
$\cP_1$ and $\cP_2$. 

A \textbf{strict 2-category} (just called 2-category) $\mathbb{A}=(A,C(a,b),K_{a,b,c},U_{a})_{a,b,c \in A}$ is given by the following data: a set $A$ of objects $a,b,c, ...$; for each ordered pair $(a,b)$ of objects of $A$, a category $C(a,b)$;
for each ordered triple $(a,b,c)$ of objects $A$, a composition functor $K_{a,b,c}:C(b,c) \times C(a,b) \to C(a,c),$
that satisfies the associativity law;
for each object $a$, a unit functor $U_a:\mathbf{1} \to C(a,a)$ where $\mathbf{1}$ is the terminal category, that provides a left and right identity for the composition functor.

A \textbf{2-stack} over the site $\bS$ is a fibered 2-category $\twoX$ over $\bS$ (i.e. a family of 2-categories indexed by objects of $\bS$, see \cite[1.10 p.29]{Breen94} for more details) such that
\begin{itemize}
  \item  2-descent is effective for objects in $\twoX$ (see \cite[1.10 p.31]{Breen94}), and
	\item  for any object $U$ of $\bS$ and for every pair of objects $X,Y$ of the 2-category $\twoX(U)$, the fibered category of arrows $\mathrm{Arr}_{\twoX(U)}(X,Y)$ of $\twoX(U)$ is a stack over $\bS_{|U}.$
\end{itemize}
 
For the notions of morphisms of 2-stacks (i.e. cartesian 2-functors), morphisms of cartesian 2-functors, modifications of 2-stacks and equivalences of 2-stacks, we refer to \cite[Chp I]{Hakim}. A \textbf{2-stack in 2-groupoids} over $\bS$ is a 2-stack $\twoX$ over $\bS$ such that for any object $U$ of $\bS$ the 2-category $\twoX(U)$ is a 2-groupoid. \emph{From now on, all 2-stacks will be 2-stacks in 2-groupoids.}

Let  $S$ be an arbitrary scheme and denote by $\bS$ the site of $S$ for a Grothendieck topology that we will fix later. We will call a stack, a Picard stack, a 2-stack over $\bS$ respectively an $S$-stack, a Picard $S$-stack, an $S$-2-stack.


\section{Recall on sheaves, gerbes and Picard stacks on a stack}\label{recall}

Let $\bS$ be a site. Let $\cX$ be a stack over $\bS$. We always assume that fibered (2-)categories come with a fixed cleavage (see \cite[\S2, \S6]{Breen10}). Deligne furnished us the following definition of site associated to a stack.

\begin{definition}\label{defsito}
The \textbf{site $\bS (\cX)$ associated to $\cX$ over $\bS$} is the site defined in the following way:
\begin{itemize}
	\item the category underlying $\bS (\cX)$ consists of the objects $(U,u)$ with $U$ an object of $\bS$ and $u$ an object of $\cX (U)$, and of the arrows $(\phi,\Phi): (U,u) \to (V,v)$ with $\phi: U \to V$ a morphism of $\bS$ and $\Phi: \phi^* v \to u$ an isomorphism in $\cX (U)$. We call the pair $(U,u)$ an \textbf{open of $\cX$} with respect to the chosen topology.
\item the topology on $\bS (\cX)$ is the one generated by the pre-topology for which a covering of $(U,u)$ is a family $\{(\phi_i, \Phi_i): (U_i,u_i) \to (U,u)\}_i $  such that the morphism of $\bS$ $ \coprod \phi_i: \coprod U_i \to U$ is a covering of $U$.
\end{itemize}
\end{definition}


\begin{definition}\label{def:SheafOnStack}
A \textbf{sheaf (of sets) $\cF$ on $\cX$} is a system $(\cF_{U,u},\theta_{\phi,\Phi})$, where for any object $(U,u)$ of $\bS (\cX)$, $\cF_{U,u}$ is a sheaf on $\bS_{|U}$, and for any arrow $(\phi,\Phi): (U,u) \to (V,v)$ of $\bS(\cX)$, $\theta_{\phi,\Phi}: \cF_{V,v} \to \phi_*  \cF_{U,u}$ is a morphism of sheaves on $\bS_{|V}$, such that 

$(i)$ if $(\phi,\Phi): (U,u) \to (V,v)$ and $(\gamma,\Gamma): (V,v) \to (W,w)$ are two arrows of $\bS(\cX)$, then $\gamma_* \theta_{\phi,\Phi} \circ \theta_{\gamma,\Gamma} = \theta_{\gamma \circ \phi, \phi^* \Gamma \circ \Phi}$;

$(ii)$ if $(\phi,\Phi): (U,u) \to (V,v)$ is an arrow of $\bS (\cX)$, the morphism of sheaves $\phi^{-1}\cF_{V,v} \to \cF_{U,u} $, obtained by adjunction from $\theta_{\phi,\Phi}$, is an isomorphism. 
\end{definition}

To simplify notations, we denote just $(\cF_{U,u})$ the sheaf $\cF= (\cF_{U,u},\theta_{\phi,\Phi})$. The set of \textbf{global sections} $\Gamma (\cX, \cF)$ of a sheaf $\cF$ on $\cX$ is the set of families $(s_{U,u})$ of sections of $\cF$ on the objects $(U,u)$ of $\bS(\cX)$ such that for any arrow $(\phi,\Phi): (U,u) \to (V,v)$ of $\bS (\cX),$ $\res_{\phi} s_{V,v}=s_{U,u}.$ 

An \textbf{abelian sheaf $\cF$ on $\cX$} is a system $(\cF_{U,u})$ 
verifying the conditions $(i)$ and $(ii)$ of Definition \ref{def:SheafOnStack}, where the $\cF_{U,u}$ 
are abelian sheaves on $\bS_{|U}$. We denote by $\Ab (\cX)$ the category of  abelian sheaves on $\cX$. According to \cite[Exp II, Prop. 6.7]{SGA4G} and \cite[Thm 1.10.1]{Tohoku}, the category $\Ab (\cX)$ is an abelian category with enough injectives.
Let ${\mathrm{R}}\Gamma(\cX,-)$ be the right derived functor of the functor $\Gamma(\cX,-): \Ab (\cX) \to \Ab$ of global sections (here $\Ab$ is the category of abelian groups). The $i$-th cohomology group $\HH^i\big({\mathrm{R}}\Gamma(\cX,-)\big)$ of ${\mathrm{R}}\Gamma(\cX,-)$ is denoted by $\HH^i(\cX,-)$.

A \textbf{stack on $\cX$} is a stack $\cY$ over $\bS$ endowed with a morphism of stacks $P: \cY \to \cX$ (called the structural morphism) such that for any object $(U,x)$ of $\bS(\cX)$ the fibered product 
$ U \times_{x,\cX,P} \cY $ is a stack over $\bS_{|U}.$

A \textbf{gerbe on $\cX$} is stack  $\cG$ over $\bS$ endowed with a morphism of stacks $P: \cG \to \cX$ (called the structural morphism) such that for any object $(U,x)$ of $\bS(\cX)$ the fibered product $ U \times_{x,\cX,P} \cG $ is a gerbe over $\bS_{|U}.$
A \textbf{morphism of gerbes on $\cX$} is a morphism of gerbes which is compatible with the underlying structural morphisms.

Let $F: \cX \to \cY$ be a morphism of $S$-stacks and let $\cG$ be a gerbe on $\cY$. The \textbf{pull-back of 
$\cG$ via $F$} is the fibered product 
\begin{equation}\label{def:pull-back}
 F^* \cG := \cX \times_{F,\cY,P} \cG 
 \end{equation}
 of $\cX$ and $\cG$ via the morphism $F: \cX \to \cY$ and the structural morphism $P: \cG \to \cY$ underlying $\cG$ (see \cite[Def 2.14]{BT} for the definition of fibered product of $S$-stacks).

 A \textbf{Picard stack on $\cX$} is a stack $\cP$ over $\bS$ endowed with a morphism of stacks $P: \cP \to \cX$ (called the structural morphism), with a morphism of stacks $ \otimes: \cP \times_{P,\cX,P} \cP \rightarrow \cP$, and with two natural isomorphisms $\ass$ and $\comm$,
such that $ U \times_{x,\cX,P} \cP $ is a Picard stack over $\bS_{|U}$ for any object $(U,x)$ of $\bS(\cX)$.

A \textbf{Picard 2-stack on $\cX$} is a 2-stack $\twoP$ over $\bS$ endowed with a morphism of 2-stacks $P: \twoP \to \cX$ (called the structural morphism - here we see $\cX$ as a 2-stack), with a morphism of 2-stacks $ \otimes: \twoP \times_{P,\cX,P} \twoP \rightarrow \twoP$, and with two natural 2-transformations $\ass$ and $\comm$,
such that $ U \times_{x,\cX,P} \twoP $ is a Picard 2-stack over $\bS_{|U}$ for any object $(U,x)$ of $\bS(\cX)$ (for more details see \cite[\S 1]{BT} or \cite{BT0}). 
An \textbf{additive 2-functor} $(F,\lambda_F): \twoP_1 \to \twoP_2$ between two Picard 2-stacks on $\cX$ is given by a morphism of 2-stacks $F: \twoP_1 \to \twoP_2$ and a natural 2-transformation $\lambda_F \colon \otimes_{\twoP_2} \circ F^2  \to F \circ \otimes_{\twoP_1}$, which are compatible with the structural morphisms of 2-stacks $P_1: \twoP_1 \to \cX$ and $P_2: \twoP_2 \to \cX$ and with the natural 2-transformations $\ass$ and $\comm$ underlying $ \twoP_1$ and $\twoP_2$.
An \textbf{equivalence of Picard 2-stacks on $\cX$} is an additive 2-functor whose underlying morphism of 2-stacks is an equivalence of 2-stacks.

Denote by $2{\bbPicard}(\cX, \bS)$ the category whose objects are Picard 2-stacks on $\cX$ and whose arrows are isomorphism classes of additive 2-functors. Applying \cite[Cor 6.5]{Tatar} to the site $\bS(\cX)$, we have the following equivalence of categories  
\begin{equation}\label{st2Stack}
 2\st: \cD^{[-2,0]}(\bS (\cX)) \longrightarrow  2{\bbPicard}(\cX, \bS).
\end{equation}
where $\cD^{[-2,0]}(\bS (\cX))$ is the derived category of length three complexes of abelian sheaves on $\cX$. Via this equivalence, Picard 2-stacks (resp. Picard stacks) on $\cX$ correspond to length three (resp. two) complexes of abelian sheaves on $\cX$. Therefore, the theory of Picard stacks is included in the theory of Picard 2-stacks.
We denote by $[\,\,]$ the inverse equivalence of $2\st$.

If $\cP$ is a Picard stack over a site $\bS$ we define its \textbf{classifying groups} $\cP^i$ for $i=1,0$ in the following way: $\cP^1$ is the group of isomorphism classes of objects of $\cP$ and
 $\cP^0$ is the group of automorphisms of the neutral object $e$ of $\cP$. 
 We define the classifying groups $\twoP^i$ for $i=2,1,0$ of a Picard 2-stack $\twoP$ over a site $\bS$ recursively:
 $\twoP^2$ is the group of equivalence classes of objects of $\twoP$,
 $\twoP^1 = {\cAut}^1(e)$ and $\twoP^0 = {\cAut}^0(e)$ where ${\cAut}(e)$ is the Picard stack of automorphisms of the neutral object $e$ of $\twoP$.
 Explicitly, $\twoP^1$ is the group of isomorphism classes of objects of ${\cAut}(e)$ and 
 $\twoP^0$ is the group of automorphisms of the neutral object of ${\cAut}(e)$.
We have the following link between the classifying groups $\twoP^i$ and the cohomology groups $\HH^{i}(\bS,[\twoP] )$ of the complex $[\twoP]$ associated to $\twoP$ via (\ref{st2Stack}):
$\twoP^i \cong \HH^{i-2}(\bS,[\twoP])$ for $i=0,1,2.$

If two Picard 2-stacks $\twoP$ and $\twoP'$ are equivalent as Picard 2-stacks, then their classifying groups are isomorphic: $ \twoP^i \cong \twoP'^i$ for $i=2,1,0$. The inverse affirmation is not true as explained in \cite[Rem 1.3]{B11}.

Let $S$ be an arbitrary scheme and denote by $\bS$ the site of $S$ for a Grothendieck topology. Let $\cX$ be an $S$-stack. A stack (resp. a Picard 2-stack) on $\cX$ will be called an $S$-stack (resp. a Picard $S$-2-stack) on $\cX$.

  \section{Gerbes with abelian band on a stack}\label{bands}

Let $F$ be an abelian sheaf on a site $\bS$.
Denote by $\bGerbe_{\bS} (F)$ the fibered 2-category of $F$-gerbes over $\bS$.

\begin{lemma}\label{lem:Gerbe2-stack}
The fibered 2-category $\bGerbe_{\bS} (F)$ of $F$-gerbes is a Picard 2-stack over $\bS$.
\end{lemma}

\begin{proof} By \cite[\S 2.6]{Breen94} the 2-descent is effective for objects  of $\bGerbe_{\bS} (F)$. Moreover, morphisms of gerbes are just morphisms of stacks and so by \cite[Examples 1.11 i)]{Breen94}, the gluing condition on arrows of $\bGerbe_{\bS} (F)$ is satisfied. Thus, the fibered 2-category $\bGerbe_{\bS} (F)$ is in fact a 2-stack over $\bS$. In \cite[Chp IV Prop 2.4.1 (i)]{Giraud} Giraud has defined the contracted product of gerbes
	(see in particular \textit{loc.cit} Example 2.4.3 for the case of gerbes bound by abelian sheaves).  He also showed that this contracted product satisfies associativity and commutativity constraints (see \cite[Chp IV Cor 2.4.2 (i) and (ii)]{Giraud}).
	Hence we can conclude that
	the contracted product of $F$-gerbes endows the 2-stack of $F$-gerbes with a Picard structure. 
	\end{proof}

\subsection{Homological interpretation of gerbes over a site}
Let $F$ be an abelian sheaf on a site $\bS$.
The \textbf{classifying groups} $\bGerbe_{\bS}^i (F)$ for $i=2,1,0$ 
of the Picard 2-stack $\bGerbe_{\bS} (F)$ are
\begin{itemize}
	\item $\bGerbe_{\bS}^2 (F)$, the abelian group of $F$-equivalence classes of $F$-gerbes;
	\item $\bGerbe_{\bS}^1 (F)$, the abelian group of isomorphism classes of morphisms of $F$-gerbes from a $F$-gerbe to itself. 
	\item $\bGerbe_{\bS}^0 (F)$, the abelian group of automorphisms of a morphism of $F$-gerbes from a $F$-gerbe to itself. 
\end{itemize}

\begin{theorem}\label{thm:H^2-general}
Let $F$ be an abelian sheaf on a site $\bS$. Then the Picard 2-stack $\bGerbe_{\bS} (F)$ of $F$-gerbes is equivalent (as Picard 2-stack) to the Picard 2-stack associated to the complex $F[2]$, where $F[2]=[F \to 0 \to 0]$ with $F$ in degree -2:
$$\bGerbe_{\bS} (F) \cong 2\st(F[2]).$$
 In particular, for $i=2,1,0,$ we have an isomorphism of abelian groups between the $i$-th classifying group $ \bGerbe^i_{\bS}(F)$ and the cohomological group $\HH^i(\bS,F).$
\end{theorem}

\begin{proof}
It is a classical result that via the equivalence of categories stated in \cite[Expos\'e XVIII, Prop 1.4.15]{SGA4D}, the complex $F[1]$ corresponds to the Picard stack $\cTors(F)$ of $F$-torsors:  $\cTors(F)= 2\st (F[1])$.
A higher dimensional analogue of the notion of torsor under an abelian sheaf is 
the notion of torsor under a Picard stack, which was introduced by Breen in \cite[Def 3.1.8]{Breen92}
and studied by the first author in \cite[\S 2]{B13} (remark that in fact in \cite{BT} the first author introduces the notion of torsor under a Picard 2-stack, see also \cite{B12}, \cite{B09} and \cite{B20}). Hence we have the notion of $\cTors(F)$-torsors.
The contracted product of torsors under a Picard 2-stack, introduced in \cite[Def 2.11]{BT}, endows the 2-stack $\bTors(\cTors(F))$ of $\cTors(F)$-torsors with a Picard structure, and by \cite[Thm 0.1]{BT} this Picard 2-stack $\bTors(\cTors(F))$ corresponds, via the equivalence of categories (\ref{st2Stack}), to the complex $[\cTors(F)][1]$:
\begin{equation}
\bTors(\cTors(F)) = 2\st (F[2]).
\label{eq:[Tors]}
\end{equation} 

In \cite[Prop 2.14]{Breen94} Breen constructs a canonical equivalence of Picard 2-stacks between the Picard 2-stack $\bGerbe_{\bS} (F)$ of $F$-gerbes and the Picard 2-stack $\bTors(\cTors(F))$ of $\cTors(F)$-torsors:
\begin{equation}
\bGerbe_{\bS} (F) \cong \bTors(\cTors(F)).
\label{eq:[Gerbe-Tors]}
\end{equation}
This equivalence and the equality (\ref{eq:[Tors]}) furnish the expected equivalence 
$\bGerbe_{\bS} (F) \cong 2\st (F[2]).$ The classifying groups of the Picard 2-stack $\bGerbe_{\bS} (F)$ are therefore
$$\bGerbe^i_{\bS}(F) \cong \HH^{i-2}(\bS,F[2]) = \HH^{i}(\bS,F).$$
\end{proof}

 \begin{remark} Via the cohomological interpretation (\ref{eq:[Tors]}) of torsors under the Picard stack of $F$-torsors, the equivalence of Picard 2-stacks (\ref{eq:[Gerbe-Tors]}) is the geometrical counterpart of the canonical isomorphism in cohomology
$\HH^{2}(\bS,F) \cong \HH^{1}(\bS,F[1]).$
\end{remark}

\subsection{Gerbes on a stack}
Let $\cX$ be a stack over a site $\bS$ and denote by $\bS (\cX)$ the site associated to $\cX$. 
Applying \cite[Chp IV]{Giraud} to the site $\bS(\cX)$,  we get the notion 
of $\cF$-gerbes on the stack $\cX$, with $\cF$ an abelian sheaf on $\cX$. 
We recall briefly this notion.

An \textbf{$\cF$-gerbe} is a gerbe $\cG$ on $\cX$ such that
 for any object $(U,x)$ of $\bS(\cX)$ the fibered product $ U \times_{x,\cX,P} \cG $ is a $\cF_{U,x}$-gerbe over $\bS_{|U}$
(here $P: \cG \to \cX$ is the structural morphism): in particular for each $i$ indexing a covering $\{U_i \to U \}_i$ of $U,$ it exists an object $g_i$ of $(U \times_{x,\cX,P} \cG)(U_i)$ and an isomorphism $\cF_{U,x |U_i} \to \uAut  (g_i)$ of sheaves of groups on $U_i$ (see \cite[Def 2.3]{Breen94}). Consider now an $\cF$-gerbe $\cG$ and an $\cF'$-gerbe $\cG'$ on $\cX$. Let $u: \cF \to \cF'$ a morphism of abelian sheaves on $\cX$. A morphism of gerbes $m: \cG \to \cG'$ is an \textbf{$u$-morphism} if $u$ is compatible with the morphism of bands 
$\mathrm{Band}(\uAut  (g)_{U,x}) \to \mathrm{Band}(\uAut  (m(g))_{U,x})$  (see \cite[Chp IV 2.1.5.1]{Giraud}). As in \cite[Chp IV Prop 2.2.6]{Giraud} an $u$-morphism $m: \cG \to \cG'$ is fully faithful if and only if $u: \cF \to \cF'$ is an isomorphism, in which case $m$ is an equivalence of gerbes. If $\cG$ and $\cG'$ are two $\cF$-gerbes on $\cX$, instead of $\id_{\cF}$-morphism $ \cG \to \cG'$ we use the terminology \textbf{$\cF$-equivalence $ \cG \to \cG'$ of $\cF$-gerbes} on $\cX$.

 $\cF$-gerbes on $\cX$ build a Picard 2-stack on $\cX$, that we denote by 
 \[ \bGerbe_{\bS(\cX)} (\cF) \]
 whose group law is given by the contracted product of $\cF$-gerbes on $\cX$ ( \cite[Chp IV 2.4.3]{Giraud}). Its neutral element is the stack $\cTors(\cF)$ of $\cF$-torsors on $\cX$, which is called the \textbf{neutral $\cF$-gerbe}.
Applying Theorem \ref{thm:H^2-general} to the abelian sheaf $\cF$ on the site $\bS (\cX)$ (see Def \ref{def:SheafOnStack}) we get

\begin{corollary}\label{thm:H^2}
We have the following equivalence of Picard 2-stacks 
$$\bGerbe_{\bS(\cX)} (\cF) \cong 2\st(\cF[2]).$$
 In particular, $ \bGerbe^i_{\bS(\cX)}(\cF) \cong \HH^i(\cX,\cF)$ for $i=2,1,0.$
\end{corollary}

\noindent Hence, $\cF$-equivalence classes of $\cF$-gerbes on $\cX$, which are the elements of the 0th-homotopy group $\bGerbe^2_{\bS(\cX)} (\cF)$, are parametrized by cohomological classes of $\HH^2(\cX,\cF)$.

\begin{remark}
$ \bGerbe_{\bS(\cX)} (\cF)$ is a Picard $\bS(\cX)$-2-stack. Via the structural morphism $F: \cX \to \bS,$ we can view $ \bGerbe_{\bS(\cX)} (\cF)$ also as a Picard $\bS$-2-stack $ \bGerbe_{\bS} (\cF)$. In this case we have that $\bGerbe_{\bS} (\cF) \cong 2\st( \tau_{\leq 0} RF_*(\cF[2]))$ where $\tau_{\leq 0}$ is the good truncation in degree 0.
We will not use this fact in the paper and therefore we omit the proof.
\end{remark}

\subsection{2-descent of $\GG_m$-gerbes}\label{descofgerbes}
We finish this section proving the effectiveness of the 2-descent of $\GG_m$-gerbes with respect to a faithfully flat morphism of schemes $p : S' \to S$ which is quasi-compact or locally of finite presentation.
We will need the \textbf{semi-local description of gerbes} done by Breen in \cite[\S 4] {Breen10} and \cite[\S 2.3]{Breen94tan}, that we recall only in the case of $\GG_{m}$-gerbes.
Denote by $\cTors(\GG_{m})$ the Picard stack of $\GG_{m}$-torsors. 
 According to Breen, to have a $\GG_{m}$-gerbe $\cG$ over a site $\bS$ is equivalent to have the data  
\begin{equation}
\big((\cTors(\GG_{m,U}),\Psi_x),(\psi_x,\xi_x) \big)_{x \in \cG(U), U \in \bS}
\label{eq:datagerbe}
\end{equation}
indexed by the objects $x$ of the $\GG_{m}$-gerbe $\cG$ (recall that $\cG$ is locally not empty), where
\begin{itemize}
	\item $\Psi_x \, : \, \cG_{\vert U}  \to \cTors(\GG_{m,U}) $ is an equivalence of $U$-stacks
	between the restriction $\cG_{\vert U}$ to $U$ of the $\GG_{m}$-gerbe $\cG$ and the neutral gerbe $\cTors(\GG_{m,U})$. This equivalence is determined by the object $x$ in $\cG(U)$, 
	
	\item $\psi_x = pr_1^* \Psi_x \circ (pr_2^* \Psi_x)^{-1}\, : \, \cTors (pr_2^*\GG_{m,U}) \to \cTors (pr_1^*\GG_{m,U})$ is an equivalence of stacks
over $U \times_S U$ (here $pr_{i}: U\times_S U \to U$ are the projections), which restricts to the identity when pulled back via the diagonal morphism $\Delta: U \to U \times_S U$, and 

	\item $\xi_x \, : \, pr_{23}^* \psi_x \circ  pr_{12}^* \psi_x \Rightarrow pr_{13}^* \psi_x $ is a isomorphism of cartesian $S$-functors between morphisms of stacks on $U\times_S U\times_S U$ (here $pr_{ij}: U\times_S U\times_S U \to U \times_S U$ are the partial projections),
which satisfies the compatibility condition 
\begin{equation}
 pr_{134}^* \xi_x \circ [pr_{34}^* \psi_x * pr_{123}^* \xi_x ] =
  pr_{124}^* \xi_x \circ [pr_{234}^* \xi_x * pr_{12}^* \psi_x ]
\label{eq:datagerbe-2arrows}
\end{equation} 
when pulled back to $U \times_S U \times_S U \times_S U :=U^4$ (here $pr_{ijk}: U^4 \to U\times_S U\times_S U$ and $pr_{ij}: U^4 \to  U\times_S U$ are the partial projections. See \cite[(6.2.7)-(6.2.8)]{Breen90} for more details).
\end{itemize}
Therefore, the $\GG_{m}$-gerbe $\cG$ may be reconstructed from the local data $(\cTors(\GG_{m}),\Psi_x)_x$ using the transition data $(\psi_x,\xi_x)$. 
We call the equivalences of stacks $\Psi_x$ the \textbf{local neutralizations} of the $\GG_{m}$-gerbe $\cG$ defined by the local objects $x \in \cG(U)$.  The transition data $(\psi_x,\xi_x)$ are in fact 2-descent data. See Appendix for this reconstruction of a $\GG_{m}$-gerbe via local neutralizations and 2-descent data.

In \S \ref{proofGTC} we will need the semi-local description of a $\GG_m$-equivalence class of a $\GG_m$-gerbe which consists in the following data: a family $(\cTors^1(\GG_{m,U}))_{U \in \bS}$ of  
 groups of isomorphism classes of $\GG_m$-torsors, bijections 
 $  \cTors^1 (pr_2^*\GG_{m,U}) \to \cTors^1 (pr_1^*\GG_{m,U})$ 
 of their pull-backs on $U\times_S U $ via the projections  $pr_{i}: U\times_S U \to U,$ and compatibility conditions 
 on the pull-back on $U\times_S U \times_S U$ of these bijections (here we use the above notations).

\begin{remark}\label{rm:-1} In this paper, Breen's semi-local description of gerbes allows us to reduce of one the degree of the cohomology groups involved: instead of working with gerbes, which are cohomology classes of $\HH^2(\bS,\GG_{m})$, we can work with torsors, which are cohomology classes of $\HH^1(\bS,\GG_{m}).$
\end{remark}

\begin{theorem}\label{thm:DescentGerbesOverSchemes}
Let $p:S' \to S$ be a faithfully flat morphism of schemes which is quasi-compact or locally of finite presentation.  To have a $\GG_{m, S}$-gerbe over $\bS$ is equivalent to have a triple  
\[(\cG',\varphi, \gamma)\]
where $\cG'$ is a $\GG_{m, S'}$-gerbe over $\bS '$ and $(\varphi, \gamma)$ are 2-descent data on $\cG'$ with respect to $p : S'\to S$. More precisely,
 \begin{itemize}
 \item  $\cG'$ is a $\GG_{m, S'}$-gerbe over ${\bS}'$,
\item   $\varphi : p_1^* \cG'  \to  p_2^* \cG'$ is  an equivalence of gerbes over $S' \times_S S'$
, where $p_{i}: S' \times_S S' \to S'$ are the natural projections,
\item  $\gamma  : \ p_{23}^*\varphi \circ p_{12}^* \varphi \Rightarrow p_{13}^* \varphi$ 
is a natural isomorphism over  $S' \times_S S' \times_S S'$,
 where  $p_{ij}:S' \times_S S' \times_S S' \to S' \times_S S'$ are the partial projections,
 \end{itemize}
 such that over $S' \times_S S' \times_S S' \times_S S'$  the compatibility condition
\begin{equation}
p_{134}^* \gamma \circ [p_{34}^* \varphi * p_{123}^* \gamma ] =
  p_{124}^* \gamma \circ [p_{234}^* \gamma * p_{12}^* \varphi ]
  \label{eq:descentgerbe-2arrows}
\end{equation} 
is satisfied, where $p_{ijk}: S' \times_S S' \times_S S' \times_S S' \times_S S' \to S' \times_S S' \times_S S'$ and $p_{ij}:  S' \times_S S' \times_S S' \times_S S' \to  S' \times_S S'$ are the partial projections.

Under this equivalence, the pull-back $p^*:\bGerbe_S (\GG_{m,S}) \to \bGerbe_{S'} (\GG_{m, S'})$ is the additive 2-functor which forgets the 2-descent data: $p^*(\cG',\varphi, \gamma)=\cG'. $
\end{theorem}

\begin{proof}  Let $(\cG',\varphi, \gamma)$ be a triplet as in the statement.
According to Appendix, the data $(\varphi, \gamma)$
satisfying the equality (\ref{eq:descentgerbe-2arrows}) are 2-descent data for the gerbe $\cG'$. As observed in Lemma \ref{lem:Gerbe2-stack}, the fibered 2-category of $\GG_{m}$-gerbes builds a 2-stack (that is, in particular, the 2-descent is effective for objects), and so $\cG'$ with its 2-descent data corresponds to a $\GG_{m,S}$-gerbe $\cG$ over $\bS$.
\end{proof}

\section{The Brauer group of a locally ringed stack}\label{brauergroup}

Let $\cX$ be a stack over a site $\bS$ and let $\bS(\cX)$ be its associated site.

A \textbf{sheaf of rings $\cA$ on $\cX$} is a system $(\cA_{U,u})$ 
verifying the conditions $(i)$ and $(ii)$ of Definition \ref{def:SheafOnStack}, where the $\cA_{U,u}$ 
are sheaves of rings on $\bS_{|U}$. Consider the sheaf of rings $\cO_\cX$ on $\cX$ given by the system $(\cO_{\cX \; U,u})$ with 
$\cO_{\cX \; U,u}$ the structural sheaf of $U$. 
The sheaf of rings $\cO_\cX$ is \textbf{the structural sheaf of the stack $\cX$} and the pair $(\cX, \cO_\cX)$  is a \textbf{ringed stack}. An \textbf{$\cO_\cX$-module $\cM$} is a system $(\cM_{U,u})$ 
verifying the conditions $(i)$ and $(ii)$ of Definition \ref{def:SheafOnStack}, where the $\cM_{U,u}$ 
are sheaves of $\cO_{U}$-modules on $\bS_{|U}$. An \textbf{$\cO_\cX$-algebra} $\cA$ is a 
system $(\cA_{U,u})$ 
verifying the conditions $(i)$ and $(ii)$ of Definition \ref{def:SheafOnStack}, where the $\cA_{U,u}$ 
are sheaves of $\cO_{U}$-algebras on $\bS_{|U}$. An $\cO_\cX$-module $\cM$ is \textbf{of finite presentation} if the $\cM_{U,u}$ are sheaves of $\cO_{U}$-modules of finite presentation.

 Now let $S$ be an arbitrary scheme and let $\bS_{\acute {e}t}$ be the \'etale site on $S$. Let $\cX=(\cX, \cO_\cX)$ be a \textbf{locally ringed $S$-stack}, i.e. a ringed stack such that, for 
 any object $(U,u)$ of the associated \'etale site $\bS_{\acute {e}t} (\cX)$, 
 and for any section $f \in \cO_{\cX \; U,u} (U)$, there exist a covering $\{(U_i,u_i) \rightarrow (U,u) \}_{i \in I}$ of $(U,u)$ such that for any $i \in I$ either $f |_{(U_i,u_i)}$ or $(1-f) |_{(U_i,u_i)}$ is invertible in $\Gamma (U_i,\cO_{\cX \; U_i,u_i}) $

An \textbf{Azumaya algebra} over $\cX$ is an $\cO_\cX$-algebra $\cA=(\cA_{U,u})$ of finite presentation as $\cO_\cX$-module which is, locally for the topology of $\bS_{\acute {e}t} (\cX)$, isomorphic to a matrix algebra, i.e. for any open $(U,u)$ of $\cX$ there exists a covering $\{(\phi_i, \Phi_i): (U_i,u_i) \to (U,u)\}_i$ in $\bS_{\acute {e}t} (\cX)$ such that $\cA _{U,u} \otimes _{\cO _{U,u}} \cO _{U_i} \cong M_{r_i}(\cO_{U_i,u_i})$ for any $i$. 
Azumaya algebras over $\cX$, that we denote by 
\[\Az(\cX),\]
 build an $S$-stack on $\cX$ by \cite[Expos\'e VIII 1.1, 1.2]{SGA1}
 (see also \cite[ (3.4.4)]{LaumonMoretB}).
Two Azumaya algebras $\cA$ and $\cA'$ over $\cX$ are \textbf{Brauer-equivalent} if there exist two locally free $\cO_\cX$-modules  $\cE$ and $\cE'$ of finite rank such that 
\[ \cA \otimes_{\cO_\cX} {\uEnd}_{\cO_\cX}(\cE) \cong \cA' \otimes_{\cO_\cX} {\uEnd}_{\cO_\cX}(\cE') .\]

\noindent The above isomorphism defines an equivalence relation because of the isomorphism of $\cO_\cX$-algebras ${\uEnd}_{\cO_\cX}(\cE) \otimes_{\cO_\cX} {\uEnd}_{\cO_\cX}(\cE') \cong {\uEnd}_{\cO_\cX}(\cE \otimes_{\cO_\cX} \cE') $. We denote by $\left[ \cA \right]$ the equivalence class of an Azumaya algebra $\cA$ over $\cX$. The set of equivalence classes of Azumaya algebra is a group under the group law given by 
$ [ \cA ] [ \cA' ] = [ \cA \otimes_{\cO_\cX} \cA' ]. $ A \textbf{trivialization} of an Azumaya algebra $\cA$ 
over $\cX$ is a couple $(\cL,a)$ with 
$\cL$ a locally free $\cO_\cX$-module and $a: \uEnd_{\cO_\cX}(\cL) \to \cA$  an isomorphism of sheaves of $\cO_\cX$-algebras. 
An Azumaya algebra $\cA$ is \textbf{trivial} if it exists a trivialization of $\cA$.
The class of any trivial Azumaya algebra is the neutral element of the above group law.
The inverse of a class $\left[ \cA \right]$ is the class $\left[ \cA^0 \right]$ with $\cA^0$ the opposite $\cO_\cX$-algebra of $\cA$.

\begin{definition}
Let $\cX= (\cX, \cO_\cX)$ be a locally ringed $S$-stack.
The \textbf{Brauer group} of $\cX$, denoted by $\Br(\cX)$, is the group of equivalence classes of 
Azumaya algebras over $\cX $. 
\end{definition}

$ \Br(-)$ is a functor from the category of locally ringed $S$-stacks (objects are locally ringed $S$-stacks and arrows are isomorphism classes of morphisms of locally ringed $S$-stacks) to the category $\Ab$ of abelian groups. Remark that the above definition generalizes to stacks the classical notion of Brauer group of a scheme: in fact if $\cX$ is a locally ringed $S$-stack associated to an $S$-scheme $X$, then $\Br(\cX)=\Br(X).$

\noindent
Consider the following sheaves of groups on $\cX$: the multiplicative group $\GG_{m, {\cX}}$, the linear general group $\GL(n, \cX),$ and the sheaf of groups $\PGL(n,\cX)$ on $\cX$ defined by the system  $(\PGL(n,\cX)_{U,u})$ where $\PGL(n,\cX)_{U,u}=\uAut\big(M_n(\cO_{\cX \; U,u})\big)$ (automorphisms of $M_n(\cO_{\cX \; U,u})$ as a sheaf of $\cO_{\cX \; U,u}$-algebras).
We have the following

\begin{lemma}\label{exseq}
Assume $n>0$. The sequence of sheaves of groups on $\cX$
\begin{equation}\label{eqno1}
1 \longrightarrow \GG_{m, {\cX}}  \longrightarrow  \GL(n, \cX) \longrightarrow  \PGL(n,\cX) \longrightarrow 1
\end{equation}
is exact.
\end{lemma}

\begin{proof} 
It is enough to show that
for any \'etale open $(U,u)$ of $\cX$, the restriction to the  \'etale site of $U$ of the sequence 
$1 \rightarrow \GG_{m_{U,u}}  \rightarrow  \GL(n)_{U,u} \rightarrow  \PGL(n)_{U,u} \rightarrow 1$  is exact and this follows by  \cite[IV, Prop. 2.3. and Cor 2.4.]{Milne80}.
\end{proof}

Let $\Lf(\cX)$ be the $S$-stack on $\cX$ of locally free $\cO_\cX$-modules. Let $\cA$ be an Azumaya algebra over $\cX$. Consider 
 the morphism of $S$-stacks on $\cX$
\begin{equation}\label{eqno2}
\End: \Lf(\cX) \longrightarrow \Az(\cX), \;\;\;\; \cL \longmapsto \uEnd_{\cO_\cX}(\cL)
\end{equation}
Following \cite[Chp IV 2.5]{Giraud}, let $\delta(\cA)$ be the fibered category over $\bS_{\acute {e}t}$
 of trivializations of $\cA$ defined in the following way:
\begin{itemize}
	\item  for any $U \in \Ob(\bS_{\acute {e}t})$, the objects of $\delta(\cA)(U)$ are trivalizations of $\cA _{|U},$ i.e. pairs 
$(\cL,a)$ with $\cL \in \Ob(\Lf(\cX)(U))$ and $ a \in \uIsom_U \big(\uEnd_{\cO_\cX}(\cL),\cA_{|U}\big)$,
\item for any arrow $f: V \to U$ of $\bS_{\acute {e}t}$, the arrows of $\delta(\cA)$ over $f$ with source $(\cL',a')$ and target $(\cL,a)$ are arrows
$\varphi : \cL ' \rightarrow \cL $ of $\Lf(\cX)$ over $f$  such that $\Az(\cX)(f) \circ a'= a \circ \End(\varphi),$ with $\Az(\cX)(f):\cA_{|V} \to \cA_{|U}.$
\end{itemize}
Since $\Lf(\cX)$ and $\Az(\cX)$ are $S$-stacks on $\cX, $  $\delta(\cA)$ is also an $S$-stack on $\cX$ (see \cite[Chp IV Prop 2.5.4 (i)]{Giraud}). Observe that the morphism of $S$-stacks $\End: \Lf(\cX) \rightarrow \Az(\cX)$ is locally surjective on objects by definition of Azumaya algebra. Moreover, it is locally surjective on arrows by exactness of the sequence (\ref{eqno1}). Therefore as in \cite[Chp IV Prop 2.5.4 (ii)]{Giraud}, $\delta(A)$ is a gerbe over $\cX$, called the \textbf{gerbe of trivializations of $\cA$}.
For any object $(U,u)$ of $\bS_{\acute {e}t} (\cX)$ the morphism of sheaves of groups on $U$
\[
(\GG_{m,{\cX}})_{U,u} = (\cO_{\cX}^*)_{U,u} \longrightarrow \big(\uAut(\cL,a)\big)_{U,u},
\]
that sends a section $g$ of $(\cO ^* _{\cX})_{U,u}$ to the multiplication $ g \cdot -:(\cL,a) _{U,u} \rightarrow (\cL,a) _{U,u}$ by this section, is an isomorphism. This means that the gerbe $\delta(\cA)$ is in fact a $\GG_{m,{\cX}}$-gerbe. By Corollary \ref{thm:H^2} we can then associate to any Azumaya algebra $\cA$ over $\cX$ a cohomological class in $\HH^2_{\acute {e}t}(\cX,\GG_{m, {\cX}})$, denoted by $\overline{\delta(\cA)}$, which is given by the $\GG_{m,{\cX}}$-equivalence class of $\delta (\cA)$ in $\bGerbe^2_S (\GG_{m, {\cX}})$.

\begin{proposition}\label{azutriv}
An Azumaya algebra $\cA$ over $\cX$ is trivial if and only if its cohomological class $\overline{\delta(\cA)} $ in $\HH^2_{\acute {e}t}(\cX,\GG_{m, {\cX}})$ is zero.
\end{proposition}

\begin{proof} The Azumaya algebra $\cA$ is trivial if and only if the gerbe $\delta (\cA)$ admits a global section if and only if its corresponding class $\overline{\delta(\cA)}$ is zero in $\HH^2_{\acute {e}t}(\cX,\GG_{m, {\cX}})$. 
\end{proof}

\begin{theorem}\label{thm:BR(X)inj}
The morphism 
\begin{eqnarray}\label{eq:delta}
\nonumber \delta: \Br(\cX) &\longrightarrow & \HH^2_{\acute {e}t}(\cX,\GG_{m, {\cX}}) \\
\nonumber  [ \cA ] & \longmapsto & \overline{\delta(\cA)}    
\end{eqnarray}
 is an injective group homomorphism.
\end{theorem}

\begin{proof} 
Let $\cA, \cB$ be two Azumaya algebras over $\cX$. For any $U \in \Ob(\bS_{\acute {e}t})$, the morphism of gerbes
\begin{eqnarray}
\nonumber  \delta (\cA)(U) \times \delta(\cB)(U) &\longrightarrow &\delta(\cA \otimes_{\cO_\cX} \cB)(U) \\
\nonumber  ((\cL,a),(\cM,b)) & \longmapsto &(\cL \otimes_{\cO_\cX} \cM, a \otimes_{\cO_\cX} b )    
\end{eqnarray}
is a $+$-morphism, where $+: \GG_{m,{\cX}} \times \GG_{m,{\cX}} \to \GG_{m,{\cX}}$ is the group law underlying the sheaf $\GG_{m,{\cX}}$.
Therefore 
\begin{equation}\label{eq:delta2}
\overline{\delta(\cA)}+ \overline{\delta(\cB)}= \overline{\delta(\cA \otimes_{\cO_\cX} \cB)}
\end{equation} 
in $\HH^2_{\acute {e}t}(\cX,\GG_{m, {\cX}})$. This equality shows first that $\overline{\delta(\cA)} = -\overline{\delta(\cA^0)} $ and also that 
$$[\cA]=[\cB] \Leftrightarrow [\cA \otimes_{\cO_\cX} \cB^0]=0  \stackrel{\mathrm{Prop}\; \ref{azutriv}}{\Leftrightarrow}  \overline{\delta(\cA \otimes_{\cO_\cX} \cB^0)} =0  \stackrel{(\ref{eq:delta2})}{\Leftrightarrow} \overline{\delta(\cA)}+ \overline{\delta(\cB^0)}=0
\Leftrightarrow \overline{\delta(\cA)}= \overline{\delta(\cB)}
 $$
These equivalences prove that the morphism 
 $ \delta: \Br(\cX) \rightarrow  \HH^2_{\acute {e}t}(\cX,\GG_{m, {\cX}})$ is well-defined and injective.
 Finally always from the equality (\ref{eq:delta2}) we get that $\delta$ is a group homomorphism.
\end{proof}

\section{Gerbes and Azumaya algebras over 1-motives}\label{geometricalcase}

Let $M=[u:X \to G]$ be a 1-motive defined over a scheme $S$, denote by $\cM$ its associated Picard $S$-stack under the equivalence constructed in \cite[Expos\'e XVIII, Prop 1.4.15]{SGA4D} and denote by $\bS(\cM)$ the site associated to the stack $\cM$ as in Definition \ref{defsito}.

\begin{definition}
\begin{enumerate}
	\item The $S$-stack of Azumaya algebras over the 1-motive $M$ is the  $S$-stack of Azumaya algebras $ \Az(\cM) $ over $\cM$.
	\item The Brauer group of the 1-motive $M$ is the Brauer group $\Br(\cM)$ of $\cM$.
      \item A $\GG_m$-gerbe on the 1-motive $M$ is a
 $\GG_{m, \cM}$-gerbe on $\cM$ (i.e. a $\GG_{m, \cM}$-gerbe on the site $\bS(\cM)$).
\end{enumerate}
\end{definition}

 By \cite[(3.4.3)]{LaumonMoretB} the associated Picard $S$-stack $\cM$ is 
isomorphic to the quotient stack $[G/X]$ (where $X$ acts on~$G$ via the given 
morphism $u : X\to G$). Note that in general it is not 
algebraic in the sense of \cite{LaumonMoretB} because it is not quasi-separated. However 
the quotient map 
\[\iota : G \longrightarrow [G/X]\cong \cM \]
 is representable, \'etale and surjective. The fiber product $G\times_{[G/X]}G$ is 
isomorphic to $X\times_S G$. Via this identification, the projections 
$q_i:G\times_{[G/X]}G \to G $ (for $i=1,2$)
correspond respectively to the second projection $p_2: X\times_S G \to G $ 
and to the map $\mu: X\times_S G \to G$ given by the action $(x,g)\mapsto 
u(x)g$. We can further identify the fiber 
product $G\times_{[G/X]}G\times_{[G/X]}G$ with $X\times_S X\times_S G$ and the 
partial projections $q_{13}, q_{23}, q_{12}: G\times_{[G/X]}G\times_{[G/X]}G 
\to G\times_{[G/X]}G$ 
respectively with the map $m_X\times \id_G :  X\times_S X \times_S G \to 
X\times_S G$ where $m_X$ denotes the group law of $X$, the map $\id_X\times \mu 
: X\times_S X \times_S G \to X\times_S 
G$, and the partial projection $p_{23}: X\times_S X \times_S G \to X\times_S 
G$. The effectiveness of the descent of Azumaya algebras with respect to the quotient map $\iota : G\to [G/X]$ is proved in the following Lemma (see 
\cite[(9.3.4)]{Olsson16} for the definition of pull-back of $\cO_\cM$-algebras):

\begin{lemma}\label{lem:DescentAzumaya}
The pull-back $\iota^*:\Az(\cM) \to \Az(G)$ is an equivalence of $S$-stacks between the $S$-stack of Azumaya algebras on $\cM$ and the $S$-stack of $X$-equivariant Azumaya algebras on $G$. More precisely,
 to have an Azumaya algebra $\cA$ on $\cM$ is equivalent 
to have a pair 
\[(A,\varphi)\]
 where $A$ is an Azumaya algebra on $G$ and 
$\varphi : p_2^* A \rightarrow \mu^*A$ is an isomorphism of Azumaya algebras on~$X\times_S G$
that satisfies (up to canonical isomorphisms) the cocycle condition
\begin{equation} \label{eq:descentdataAzumaya}
(m_X\times id_G)^* \varphi=\big((\id_X\times
\mu)^*\varphi\big) \circ \big((p_{23})^* \varphi\big).
\end{equation}
Under this equivalence, the pull-back $\iota^*:\Az(\cM) \to \Az(G)$ is the morphism of stacks which forgets the descent datum: $\iota^*(A,\varphi)=A. $
\end{lemma}

\begin{proof} Since the assertion is local for the topology on $\bS_{\acute {e}t}(\cM) ,$ it suffices to prove it for any open $(U,u) $ of $\cM$, where $U$ is an object of $\bS_{\acute {e}t}$ and $x$ is an objecy of $\cM(U)$. The descent of quasi-coherent modules is known for the morphism $\iota_U : G \times_{\iota,\cM,x} U \to U$ obtained by base 
change (see \cite[Thm (13.5.5)]{LaumonMoretB}). The additional algebra structure descends by \cite[II Thm 3.4]{KO74}. Finally the Azumaya algebra structure descends by \cite[III, Prop 2.8]{J14}. 
Since an Azumaya algebra on $\cM$ is by definition a collection of Azumaya algebras on the various schemes $U$, the general case follows. 
\end{proof}

Now we prove also the effectiveness of the 2-descent of $\GG_m$-gerbes under the quotient map $\iota : G \to \cM ,$ using the result of Section \ref{descofgerbes}.

\begin{lemma}\label{lem:DescentGerbes}
To have a $\GG_{m, \cM}$-gerbe $\cG$ on $\cM$ is equivalent to have a triplet 
$$(\cG',\varphi, \gamma)$$
where $\cG'$ is a $\GG_{m, G}$-gerbe on $G$  and $(\varphi, \gamma)$ are 2-descent data on $\cG'$ with respect to $\iota : G\to [G/X]$, that is 
\begin{itemize}
	
	\item $\varphi: p_2^* \cG' \rightarrow \mu^* \cG'$ is an equivalence of gerbes on~$X\times_S G$, 
	
	\item $\gamma: \big((\id_X\times
\mu)^*\varphi\big) \circ \big((p_{23})^* \varphi\big) \Rightarrow (m_X\times id_G)^* \varphi$ is a natural isomorphism on $X\times_S X \times_S G  \cong G\times_{[G/X]}G\times_{[G/X]}G $, 
\end{itemize}
which satisfies the compatibility condition
\begin{equation} \label{eq:descentdataGerbe}
p_{134}^* \gamma \circ [p_{34}^* \varphi * p_{123}^* \gamma ] =
  p_{124}^* \gamma \circ [p_{234}^* \gamma * p_{12}^* \varphi ]
\end{equation}
when pulled back to 
$ X\times_S X\times_S X \times_S G  \cong G\times_{[G/X]} G\times_{[G/X]}G\times_{[G/X]}G :=G^4$ (here $pr_{ijk}: G^4 \to G\times_{[G/X]} G\times_{[G/X]}G$ and $pr_{ij}: G^4 \to  G\times_{[G/X]} G$ are the partial projections).
\end{lemma}

\begin{proof} 
A $\GG_{m, \cM}$-gerbe on $\cM$ is by definition a collection of $\GG_{m,U}$-gerbes over the various 
objects $U$ of $\bS$. Hence it is enough to prove that for any object $U$ of $\bS$ and any object $x$ of $\cM(U)$, the 2-descent of $\GG_{m}$-gerbes with respect to the morphism $\iota_U : G \times_{\iota,\cM,x} U \to U$ obtained by base 
change is effective. But this is a consequence of Theorem \ref{thm:DescentGerbesOverSchemes}.
\end{proof}


  \section{The generalized Theorem of the Cube for 1-motives and its consequences}\label{proofGTC}

We use the same notation of the previous Section.
We denote by $\cM^3 = \cM \times_{\bS} \cM \times_{\bS} \cM $ (resp. $\cM^2 = \cM \times_{\bS} \cM$) the fibered product of 3 (resp. 2) copies of $\cM.$ Since any Picard stack admits a global neutral object, it exists a unit section denoted by $\epsilon: \bS \rightarrow \cM$. 
Consider the map
  \[ s_{ij}:=   \cM \times_{\bS} \cM  \to \cM \times_{\bS} \cM \times_{\bS} \cM  \] 
which inserts the unit section $\epsilon: {\bS} \rightarrow \cM$ into the $k$-th factor for $k \in \{1,2,3 \}-\{i,j\} .$ If $\ell$ is a prime number and $\HH$ is an abelian group, $\HH(\ell)$ denotes the $\ell$-primary component of $\HH$.

\begin{definition}\label{def-thmCube}
Let $M$ be a 1-motive defined over a scheme $S$. Let $\ell$ be a prime number distinct from the residue characteristics of $S$.
The 1-motive $M$ \textbf{satisfies the generalized Theorem of the Cube for the prime $\ell$} if the natural homomorphism 
\begin{equation}\label{cube}
\begin{matrix}
\prod s_{ij}^*:\HH_{\acute {e}t}^2(\cM^3,\GG_{m, \cM^3})(\ell) & \longrightarrow &  
\big( \HH_{\acute {e}t}^2(\cM^2 ,\GG_{m, \cM^2})(\ell) \big)^3\\
x &  \longmapsto & (s_{12}^* (x), s_{13}^* (x),s_{23}^* (x))
\end{matrix}
\end{equation}
 is injective.
\end{definition}

\subsection{Its consequences}
\begin{proposition}\label{prop:Isogeny-thmCube}
Let $M$ be a 1-motive satisfying the generalized Theorem of the Cube for a prime $\ell$ distinct from the residue characteristics of $S$. Let $N: \cM \rightarrow \cM$ be the multiplication by $N$ on the Picard $S$-stack $\cM$. Then for any $y \in \HH_{\acute {e}t}^2(\cM,\GG_{m, \cM})(\ell) $ we have that 
\begin{equation}\label{N*}
 N^*(y)=N^2y+\Big(\frac{N^2-N}{2}\Big)\big((-id_\cM)^*(y)-y\big).
\end{equation}
\end{proposition}

\begin{proof}
First we prove that given three contravariant functors $F,G,H : \cP \rightarrow \cM$,
we have the following equality for any $y$ in $\HH^2_{\acute {e}t}(\cM,\GG_{m, \cM})(\ell)$
\begin{equation}\label{eq:f+g+h}
(F+G+H)^*(y)-(F+G)^*(y)-(F+H)^*(y)-(G+H)^*(y)+F^*(y)+G^*(y)+H^*(y)=0.
\end{equation}
Let $pr_i : \cM \times \cM \times \cM \rightarrow \cM $ the projection onto the $i^{th}$ factor. Put $m_{i,j}=pr_i \otimes pr_j : \cM \times \cM \times \cM \rightarrow \cM $ and $m=pr_1 \otimes pr_2 \otimes pr_3 : \cM \times \cM \times \cM \rightarrow \cM ,$ where $\otimes$ is the law group of the Picard $S$-stack $\cM$.
The element
\[
z=m^*(y)-m_{1,2}^*(y)- m_{1,3}^*(y) -m_{2,3}^*(y)+pr_1^*(y)+pr_2^*(y)+pr_3^*(y)
\]
of $\HH_{\acute {e}t}^2(\cM^3,\GG_{m, \cM ^3})(\ell)$ is zero when restricted to $S \times \cM \times \cM ,\,  \cM \times S \times \cM$ and $\cM \times \cM \times S$ (this restriction is obtained inserting the unit section $\epsilon: S \rightarrow \cM$). Thus it is zero in  $\HH_{\acute {e}t}^2(\cM^3,\GG_{m, \cM^3})(\ell)$ by the generalized Theorem of the Cube for $\ell$. Finally, pulling back $z$ by $(F,G,H): \cP \to \cM \times \cM \times \cM $ we get (\ref{eq:f+g+h}). 

Now, setting $F=N,\,G=id _\cM, \, h=(-id_\cM)$ we get 
\[
N^*(y)=(N+id _\cM)^*(y)+(N-id _\cM)^*(y)+0^*(y)- N^*(y)-(id _\cM)^*(y)- (-id _\cM)^*(y).
\]
We rewrite this as
\[
(N+id _\cM)^*(y)- N^*(y)=N^*(y) -(N-id _\cM)^*(y)+(id _\cM)^*(y)+ (-id _\cM)^*(y).
\]
If we denote $z_1=y$ and $z_N=N^*(y)-(N-id _\cM)^*(y)$, we obtain $z_{N+1}=z_N + y+ (-id _\cM)^*(y)$.
By induction, we get $z_N=y+(N-id _\cM)(y + (-id _\cM)^*(y))$. From the equality $N^*(y)=z_N + (N-id _\cM)^*(y)$ we have
\[
N^*(y)= z_N + z_{N-1} + \dots +z_1 . 
\]
and therefore we are done.
\end{proof}

\begin{corollary}\label{cor-thmCube}
Let $M$ be a 1-motive satisfying the generalized Theorem of the Cube for a prime $\ell$. Then, if $\ell \not=2$, the $\ell^n$-torsion elements of
$
\HH_{\acute {e}t}^2(\cM,\GG_{m, \cM})
$
are contained in

$$ \ker\big[ (\ell_\cM^n)^*: 
\HH_{\acute{e}t}^2(\cM,\GG_{m, \cM})  \longrightarrow \HH_{\acute {e}t}^2(\cM,\GG_{m, \cM}) \big] $$
and if $\ell =2$, they are contained in

$$ \ker\big[ (2_\cM^{n+1})^*: 
\HH_{\acute{e}t}^2(\cM,\GG_{m, \cM})  \longrightarrow \HH_{\acute {e}t}^2(\cM,\GG_{m, \cM}) \big]. $$
\end{corollary}

\begin{proof}
The result follows by (\ref{N*}).
\end{proof}

\subsection{Its proof}
We finish this section by searching the hypothesis we should put on the base scheme $S$ in order to have that the 1-motive $M=[X \stackrel{u}{\to} G]$ satisfies the generalized Theorem of the Cube.
From now on we will switch freely between the two equivalent notion of invertible sheaf $\mathcal{L}$ on the extension $G$ and $\GG_m$-torsor $\uIsom(\mathcal{O}_G,\mathcal{L})$ on $G$.
The extension $G$ fits into the following short exact sequence

\[0 \longrightarrow T \longrightarrow	G \stackrel{\pi}{\longrightarrow}  
A \longrightarrow 0 
\]
The pull-back of gerbes defined in (\ref{def:pull-back}) allows us to define an homomorphism of abelian groups $\pi^*:\bGerbe_S^2 (\GG_{m, A})  \to \bGerbe_S^2 (\GG_{m, G}) .$

\begin{proposition}\label{cor:IsoOnTorsionG-A}
Let $S$ be a normal scheme. Let $G$ be an extension of an abelian $S$-scheme $A$ by $\GG_m^{r}.$ The pull-back $\pi^*:\bGerbe_S^2 (\GG_{m, A})  \to \bGerbe_S^2 (\GG_{m, G}) $ is surjective. 
\end{proposition}

\begin{proof}
	Denote by $\cTorsRig(\GG_{m, G}) $ the Picard $S$-stack of $\GG_m$-torsors on $G$ with rigidification along the unit section $\epsilon_G: S \to G$. Because of this unit section $\epsilon_G$, the group of isomorphism classes of $\GG_m$-torsors over $G$ with rigidification is canonically isomorphic to the quotient of the group of isomorphism classes of $\GG_m$-torsors over $G$ by the group of isomorphism classes of $\GG_m$-torsors over $S$:
	\begin{equation} \label{eq:TORS-TORSRIG}
	\cTorsRig^1(\GG_{m, G}) \cong \cTors^1(\GG_{m, G}) / \cTors^1(\GG_{m, S}).
	\end{equation}

Denote by $\cCub(G,\GG_m)$ the Picard $S$-stack of $\GG_m$-torsors on $G$ with cubical structure and by $\cCub^i(G,\GG_m)$ for $i=1,0$ its classifying groups. 
Roughly speaking a $\GG_m$-torsor on $G$ is cubical if it satisfies the Theorem of the Cube, for details see \cite[Def 2.2]{Breen83}.
	In \cite[Prop 2.4]{Breen83}, Breen proves the Theorem of the Cube for extensions of abelian schemes by tori which are defined over a normal scheme, that is the forgetful additive functor
	$\cCub(G,\GG_m) \to \cTorsRig(\GG_{m, G})$ is an equivalence of Picard $S$-stacks. In particular
	\begin{equation} \label{eq:CUB-TORSRIG}
	\cCub^1(G,\GG_m) \cong \cTorsRig^1(\GG_{m, G}).
	\end{equation}
	With our hypothesis, by \cite[Chp I, Rem 7.2.4]{MoretBailly}, any $\GG_m$-torsors on $G$ with cubical structure comes from a $\GG_m$-torsors on $A$ with cubical structure. Using the isomorphisms (\ref{eq:TORS-TORSRIG}) and (\ref{eq:CUB-TORSRIG}), we get the surjection
	\begin{equation}
	\pi^* : \cTors^1(\GG_{m, A}) \longrightarrow  
	\cTors^1(\GG_{m, G}).
	\label{eq:surjection-Tors}
	\end{equation}
	Now let $ \cG$ be a $\GG_{m, G}$-gerbe on $G.$
	Breen's semi-local description of gerbes (\ref{eq:datagerbe}) asserts that to have  $ \cG$ is equivalent to have the local data $(\cTors(\GG_{m,G,U}),\Psi_x )_{x \in \cG(U), U \in 
		\bS_{fppf}}$ endowed with the transition data
	$(\psi_x,\xi_x)$. Let $y=\pi( P(x)) \in A(U)$, where $P: \cG \to G$ is the structural morphism of $\cG$.
	The equivalence of $U \times_S U$-stacks $\psi_x :  \cTors (pr_2^*\GG_{m,G,U}) \to \cTors (pr_1^*\GG_{m,G,U})$ induces a bijection between the classifying groups $\psi_x^1: 
	\cTors^1 (pr_2^*\GG_{m,G,U})  \to  \cTors^1 (pr_1^*\GG_{m,G,U}) .$
	Because of the surjection (\ref{eq:surjection-Tors}), all torsors over $G$ come from torsors over $A$ up to isomorphisms, and so 
	we have a bijection
	\[ \psi_y^1: \cTors^1 (pr_2^*\GG_{m,A,U}) \to  \cTors^1 (pr_1^*\GG_{m,A,U}) \]
	such that $ \psi_y^1 = (\pi^*)^* \psi_x^1.$ By pull-back via (\ref{eq:surjection-Tors}),
	the isomorphism of cartesian $S$-functors  $\xi_x : \, pr_{23}^* \psi_x \circ  pr_{12}^* \psi_x \Rightarrow pr_{13}^* \psi_x$
	gives $\xi_y^1 = (\pi^*)^* \xi^1_x: \, pr_{23}^* \psi_y^1 \circ  pr_{12}^* \psi_y^1 \Rightarrow pr_{13}^* \psi_y^1$ over $U \times_S U \times_S U$, which satisfies an analogue of the equality (\ref{eq:datagerbe-2arrows}).
	The local data $(\cTors^1(\GG_{m,A,U}))_{y=\pi( P(x)), x \in \cG(U)}$ endowed with the transition data
	$(\psi_y^1,\xi_y^1)$, which are defined on isomorphism classes of $\GG_m$-torsors, furnish a $ \GG_{m,A}$-equivalence class $\overline{\cG}'$ of a $ \GG_{m,A}$-gerbe on $A$ such that $\pi^*(\overline{\cG}')=\overline{\cG}$.  
\end{proof}

\noindent
We give an immediate application of this result in the case of extensions over a field $k.$ 

\begin{corollary}\label{Br(G)}
Let $G$ be an extension of an abelian variety by a torus over a field $k$. Then $ \Br(G) \cong \HH^2_{\acute {e}t}(G, \GG_{m,G}).$
\end{corollary}

\begin{proof} By Gabber's unpublished result \cite{deJong}, if $A$ is an abelian variety 
defined over a field $k$, then $\Br(A) \cong \HH^2_{\acute {e}t}(A, \GG_{m,A})$.
We have the following commutative diagram
 \begin{equation}
\xymatrix{
 \Br(A) \ar[d]_{\pi^*} \ar[r]^{\cong \;\;\;\;\;\;}  &  \HH^2_{\acute {e}t}(A,\GG_m) \ar[d]^{\pi^*}\\ 
\Br(G) \ar@{^{(}->}[r]_{\delta \;\;\;\;\;\;}  &  \HH^2_{\acute {e}t}(G, \GG_m) 
}	
\end{equation}
where $\pi : G \to A $ is the surjective morphism of varieties underlying $G$ and $\pi^*$ denotes the pull-back maps of Azumaya algebras and cohomological classes.
By \cite[II, Prop 1.4]{Grothendieck68} the cohomological groups $\HH^2_{\acute {e}t}(G, \GG_m)$ and $\HH^2_{\acute {e}t}(A, \GG_m)$ are torsion groups and so Theorem \ref{thm:H^2-general} and Proposition \ref{cor:IsoOnTorsionG-A} imply that $\pi^* : \HH^2_{\acute {e}t}(A,\GG_m) \to  \HH^2_{\acute {e}t}(G,\GG_m) $ is surjective. Hence the injective homomorphism $\delta$ on the bottom row is surjective too. 
\end{proof}

 Let $G_i$ be an extension of an abelian $S$-scheme $A_i$ by an $S$-torus for $i=1,2,3$, and denote by $\epsilon_i^G: S \to G_i$ its unit section.
Let $ s_{ij}^G:=  G_i \times_S G_j  \to  G_1 \times_S G_2 \times_S G_3 $ be the map obtained from the unit section $\epsilon_k^G: S \rightarrow G_k$ after the base change $G_i\times_S G_j \to S$ (i.e. the map 
which inserts $\epsilon_k^G: S \rightarrow G_k$ into the $k$-th factor for $k \in \{1,2,3 \}-\{i,j\} $).

\begin{corollary}\label{cor:HythForThmCubeOnG}
Let $S$ be a connected, reduced, normal and noetherian scheme and 
let $G_i$ be an extension of an abelian $S$-scheme $A_i$ by $\GG_m^{r_i}$ for $i=1,2,3$. 
 Let $\ell$ be a prime distinct from the residue characteristics of $S$.
Then, with the above notation, the natural homomorphism 
\begin{equation}\label{eq:cube}
\begin{matrix}
\prod s_{ij}^{G*} : \HH_{\acute {e}t}^2(G_1 \times_S G_2 \times_S G_3,\GG_m)(\ell) & \longrightarrow &  \prod_{(i,j) \in \{1,2,3 \}}
 \HH_{\acute {e}t}^2(G_i \times_S G_j ,\GG_m)(\ell)\\
x &  \longmapsto & (s_{12}^{G*} (x), s_{13}^{G*} (x),s_{23}^{G*} (x))
\end{matrix}
\end{equation}
 is injective.

In particular, if the base scheme is connected, reduced, normal and noetherian, an extension of an abelian $S$-scheme by $\GG_m^{r}$ satisfies the generalized Theorem of the Cube for any prime $\ell$ distinct from the residue characteristics of $S$.
\end{corollary}

\begin{proof}
Let $p^G_{ij}  :  G_1 \times_S G_2 \times_S G_3 \rightarrow G_i \times_S G_j $ be the projection maps for $i,j=1,2,3$ and let $\alpha ^G$ be the map

\begin{equation}
\begin{matrix}
\alpha ^G \, : \, \prod_{(i,j) \in \{1,2,3 \}}
 \HH_{\acute {e}t}^2(G_i \times_S G_j ,\GG_m)(\ell) & \longrightarrow & \HH_{\acute {e}t}^2(G_1 \times_S G_2 \times_S G_3,\GG_m) (\ell) \\
((y_2,y_3),(y_1,y_3),(y_1,y_2)) &  \longmapsto & \sum _{i,j=1,2,3} p_{ij}^{G*}(y_i,y_j))
\end{matrix}
\end{equation}
\noindent
Analogously we define the maps $\alpha ^A$ and $p^A_{ij}$. We observe that the injectivity of the map 
$\prod s_{ij}^{G*}$ is equivalent to the surjectivity of the map $\alpha^G$ (\cite[Rem page 55]{Mu}).
If we denote  $\pi_i : G_i \to A_i$ the surjections underlying the extensions $G_i$ (for $i=1,2,3$),
$
(\pi_i \times \pi_j) \circ p^G_{ij} = p^A_{ij} \circ (\pi_1 \times \pi_2 \times \pi_2)$
and so we have the following commutative diagram 
\begin{equation}
\xymatrix{
\prod_{(i,j) \in \{1,2,3 \}} \HH_{\acute {e}t}^2(A_i \times_S A_j ,\GG_m)(\ell)\;  \ar@{>>}[d]\ar@{>>}[r]^{\;\;\;\;\;\alpha ^A}&  \;  
\HH_{\acute {e}t}^2(A_1 \times_S A_2 \times_S A_3,\GG_m)(\ell)	\ar@{>>}[d]
\\
\prod_{(i,j) \in \{1,2,3 \}} \HH_{\acute {e}t}^2(G_i \times_S G_j ,\GG_m)(\ell) \ar[r]^{\;\;\;\;\; \alpha ^G} &  \;  
 \HH_{\acute {e}t}^2(G_1 \times_S G_2 \times_S G_3,\GG_m)(\ell)}
\end{equation}
where the vertical arrows, which are the pull-backs induced by the $\pi_i$, are surjective by 
Proposition \ref{cor:IsoOnTorsionG-A} and Corollary \ref{thm:H^2}. 
The top horizontal arrow is surjective since under our hypotheses abelian $S$-schemes satisfy the generalized Theorem of the Cube (see \cite[Cor 2.6]{Hoobler72}). 
Hence we can conclude that also the down horizontal arrow
is surjective, i.e. $\prod s_{ij}^{G*}$ is injective.
\end{proof}

Using the effectiveness of the 2-descent for $\GG_m$-gerbes via the quotient map $\iota: G \to \cM$ (Lemma \ref{lem:DescentGerbes}), from the above corollary we get

\begin{theorem}\label{thm:HythForThmCubeOnM}
1-motives, which are defined over a connected, reduced, normal and noetherian scheme $S$, and whose underlying tori are split, satisfy the generalized Theorem of the Cube for any prime $\ell$ distinct from the residue characteristics of $S$.
\end{theorem}


\section{Cohomological classes of 1-motives which are Azumaya algebras}\label{proofmainTHM}

Let $S$ be a scheme. We will need the finite site on $S$: first recall that a morphism of schemes $f: X \rightarrow S$ is said to be \textbf{finite locally free} if it is finite and $f_*(\cO_X)$ is a locally free $\cO_S$-module. In particular, by \cite[Prop (18.2.3)]{EGAIV} finite \'etale morphisms are finite locally free. 
The finite site on $S$, denoted 
$\bS_{f}$, is the category of finite locally free schemes over $S$, endowed with 
the topology generated from the pretopology for which the set of coverings of a finite locally free scheme $T$ over $S$ is the set of single morphisms $u: T' \rightarrow T$ such that $u$ is finite locally free and $T=u(T')$ (set theoretically). There is a morphism of site $\tau:\bS_{fppf} \rightarrow \bS_{f}$. 
If $F$ is a sheaf for the \'etale topology, then we define as in \cite[\S 3]{Hoobler72}
$$F(T)_f :=\big\{ y \in F(T) \mid  \mathrm{\;there \; is\; a \;covering\;} u:T' \rightarrow T \mathrm{\; in \;} \bS_f \mathrm{\; with \;} F(u)(y)=0\big\}$$
i.e. $F(T)_f$ are the elements of $F(T)$ which can be split by a finite locally free covering.

\begin{proposition}\label{mainThmG}
Let $G$ be an extension of an abelian scheme by a torus, which is defined over a normal and noetherian scheme $S$, and which satisfies the generalized Theorem of the Cube for a prime number $\ell$ distinct from the 
 characteristics of $S$. Then the $\ell$-primary component of the kernel of the homomorphism $\HH^2_{\acute {e}t}(\epsilon):\HH^2_{\acute {e}t}(G,\GG_{m,G}))\rightarrow \HH^2_{\acute {e}t}(S,\GG_{m,S})$ induced by the unit section 
$\epsilon : S \rightarrow G$ of $G,$ is contained in the Brauer group of $G$:
$$\ker \big[\HH^2_{\acute {e}t}(\epsilon):\HH^2_{\acute {e}t}(G,\GG_{m,G}))\longrightarrow 
\HH^2_{\acute {e}t}(S,\GG_{m,S})\big] (\ell) \subseteq \Br(G).$$ 
\end{proposition}

\begin{proof} In order to simplify notations we denote by $ \ker (\HH^2_{\acute {e}t}(\epsilon))$ the kernel of the homomorphism $\HH^2_{\acute {e}t}(\epsilon):\HH^2_{\acute {e}t}(G,\GG_{m,G})\rightarrow \HH^2_{\acute {e}t}(S,\GG_{m,S}).$ We identify \'etale and fppf cohomologies for the smooth group schemes $\mu _{\ell^n}$ and $\GG _m.$

(1) First we show that $\HH^2_f(G, \tau_*\mu_{\ell^n})$ is isomorphic to $\HH^2_{\acute {e}t}(G,\mu_{\ell^n})_f$.
By definition, $\R^1 \tau_*\mu_{\ell^n}$ is the sheaf on $G$ associated to the presheaf $U \rightarrow 
\HH^1(U_{fppf},\mu_{\ell^n})$. This latter group classifies torsors in $U_{fppf}$ under the finite locally free group scheme $\mu_{\ell^n}$.
Since any $\mu_{\ell^n}$-torsor is trivialized by using itself as an $f$-cover,
 we have that $\R^1 \tau_*\mu_{\ell^n}=(0)$. The Leray spectral sequence for the morphism of sites
$\tau:\bS_{fppf} \rightarrow \bS_{f}$ (see \cite[page 309]{Milne80}) gives then the isomorphism 
$$
 \HH^2_f(G, \tau_*\mu_{\ell^n}) \cong \ker \big[\HH^2_{\acute {e}t}(G,\mu_{\ell^n}) \stackrel{\pi}{\longrightarrow} 
\HH^0_f
(G,\R^2 \tau_* \mu_{\ell^n})\big] = \HH^2_{\acute {e}t}(G,\mu_{\ell^n})_f
$$
where the map $\pi$ is the edge morphism which can be interpreted as the canonical morphism from the presheaf $U \rightarrow \HH^2_{\acute {e}t}(U_{fppf} ,\mu_{\ell^n}) $ to the associated sheaf $\R^2 \tau_* \mu_{\ell^n}$.

(2) Now we prove that $\HH^2_{\acute {e}t}(G,\mu_{\ell^\infty})_f$ maps onto
$ \ker (\HH^2_{\acute {e}t}(\epsilon))(\ell).$
Let $x$ be an element of $ \ker (\HH^2_{\acute {e}t}(\epsilon))$ with $\ell^n x=0$ for some $n$. The filtration on the Leray spectral sequence for 
$\tau:\bS_{fppf} \rightarrow \bS_{f}$ and the Kummer sequence give the following exact commutative diagram
\begin{equation}
\xymatrix{ & & \Pic(G) \ar[r]^{\pi' \qquad \qquad} \ar[d]_{d} & \HH^0_f(G,(\R^1 \tau_*\GG_{m})_{\ell^n}) \ar[d]_{d'}\\
 0 \ar[r] & \HH^2_{\acute {e}t}(G,\mu_{\ell^n})_f \; \ar[r] \ar[d]& \HH^2_{\acute {e}t}(G,\mu_{\ell^n}) \ar[r]^{\pi \qquad} \ar[d]_{i} & \HH^0_f(G,\R^2 \tau_* \mu_{\ell^n}) \ar[d]\\
0 \ar[r] &\HH^2_{\acute {e}t}(G,\GG_{m})_f \; \ar[r] & \HH^2_{\acute {e}t}(G,\GG_{m}) \ar[r] \ar[d]_{\ell^n} & \HH^0_f(G, \R^2 \tau_*\GG_{m} ) \\
& & \HH^2_{\acute {e}t}(G,\GG_{m,G}) & 
}	
\end{equation}
where $(\R^1 \tau_*\GG_{m})_{\ell^n}$ is the cokernel of the multiplication by $\ell^n.$ Since $\ell^n x=0$, we can choose an $y \in \HH^2_{\acute {e}t}(G,\mu_{\ell^n})$ such that $i(y)=x.$ By Corollary \ref{cor-thmCube}, the isogeny $\ell^{2n}:G \to G$ is a finite locally free covering which splits $x$, that is $x \in \HH^2_{\acute {e}t}(G,\GG_{m})_f.$ Moreover $i( (l^{2n})^* y)=  (l^{2n})^* i(y)=0$ and so there exists an element $z \in  \Pic(G)$ such that $d(z)=  (l^{2n})^* y.$ In particular $d'(\pi'(z))=  \pi((l^{2n})^* y)$ is an element of $\HH^0_f(G,\R^2 \tau_* \mu_{\ell^n}).$ By the Theorem of the Cube for the extension $G$ (see \cite[Prop 2.4]{Breen83}), we have that $(l^{2n})^* z= l^{2n} z'$ for some 
$z' \in  \Pic(G)$, which implies that $\pi' (z)=0$ in $(\R^1 \tau_*\GG_{m})_{\ell^n}(G_f)$. From the equality $ \pi((l^{2n})^* y) = d'(\pi'(z))= 0$ follows $\pi (y)=0$, which means that $y$ is an element of $\HH^2_{\acute {e}t}(G,\mu_{\ell^n})_f.$ 

(3) Here we show that $\ker (\HH^2_{\acute {e}t}(\epsilon))(\ell) \subseteq \tau^* \HH^2_f(G,\GG_{m})$.
By the first two steps, $\HH^2_f(G,\tau_* \mu_{\ell^n})$ maps onto
$ \ker (\HH^2_{\acute {e}t}(\epsilon))(\ell).$
Since $\HH^2_f(G, \tau_*\mu_{\ell^n}) \subseteq \HH^2_f(G, \tau_*\GG_{m})$, we can then conclude that $\ker (\HH^2_{\acute {e}t}(\epsilon))(\ell) \subseteq \tau^* \HH^2_f(G,\GG_{m})$.

(4) Let $x$ be an element of $ \ker (\HH^2_{\acute {e}t}(\epsilon))$ with $\ell^n x=0$ for some $n$.
Using the morphism of sites $\tau:\bS_{fppf} \rightarrow \bS_{f},\,$ in  \cite[Lem 3.2]{Hoobler72} Hoobler built a commutative diagram where upper and lower lines  give rise to spectral sequences comparing $\check{\mathrm{C}}$ech and sheaf cohomology for $\bS_{fppf}$ and $\bS_{f}$. Using this diagram he showed that the group $ \tau^* \HH^2_f(G,\GG_{m})$ is contained in the $\check{\mathrm{C}}$ech cohomology group $\check{\HH}^2_{\acute {e}t}(G,\GG_m)$. 
Thus, from step (3), $x$ is an element of $\check{\HH}^2_{\acute {e}t}(G,\GG_m)$. 
By Corollary \ref{cor-thmCube}, the isogeny $\ell^{2n}:G \to G$ is a finite locally free covering which splits $x$, that is $x \in \check{\HH}^2_{\acute {e}t}(G,\GG_m)_f.$ In fact $x $ is an element of $ \ker \big[\check{\HH}^2_{\acute {e}t}(\epsilon):\check{\HH}^2_{\acute {e}t}(G,\GG_m) \to 
\check{\HH}^2_{\acute {e}t}(S,\GG_m)\big]_f$. Finally by \cite[Prop 3.1.]{Hoobler72} we can then conclude that $x$ is an element of $\Br(G).$
\end{proof}

Finally we can prove

\begin{theorem}\label{mainTHM}
Let $M=[u:X \rightarrow G]$ be a 1-motive defined over a normal and noetherian scheme $S$. Assume that the extension $G$ underlying $M$ satisfies the generalized Theorem of the Cube for a prime number $\ell$ distinct from the residue characteristics of $S$. Then the $\ell$-primary component of the kernel of the homomorphism $\HH^2_{\acute {e}t}(\epsilon):\HH^2_{\acute {e}t}(\cM,\GG_{m,\cM})\rightarrow \HH^2_{\acute {e}t}(S,\GG_{m,S})$ induced by the unit section 
$\epsilon : S \rightarrow M$ of $M,$ is contained in the Brauer group of $M$:
$$\ker \big[\HH^2_{\acute {e}t}(\epsilon):\HH^2_{\acute {e}t}(\cM,\GG_{m,\cM})\longrightarrow \HH^2_{\acute {e}t}(S,\GG_{m,S})\big] (\ell) \subseteq \Br(\cM).$$ 
\end{theorem}

\begin{proof}
We have to show that if $G$ satisfies the generalized Theorem of the Cube for a prime $\ell$, then
\[
\ker \big[\HH^2_{\acute {e}t}(\epsilon):\HH^2_{\acute {e}t}(\cM,\GG_{m,\cM}))\longrightarrow 
\HH^2_{\acute {e}t}(S,\GG_{m,S})\big](\ell) \subseteq \Br(\cM).
\] 
In order to simplify notations we denote by $ \ker (\HH^2_{\acute {e}t}(\epsilon))$ the kernel of the homomorphism $\HH^2_{\acute {e}t}(\epsilon):\HH^2_{\acute {e}t}(\cM,\GG_{m,\cM}))\rightarrow \HH^2_{\acute {e}t}(S,\GG_{m,S}).$ Let $x$ be an element of $\ker (\HH^2_{\acute {e}t}(\epsilon))$ such that $\ell^n x =0$ for some $n$. Let $y= \iota^* x$ the image of $x$ via the homomorphism $\iota^*: \HH^2_{\acute {e}t}(\cM,\GG_{m,\cM}) \to \HH^2_{\acute {e}t}(G,\GG_{m,G})$ induced by the quotient map $\iota: G\to [G/X]$. Because of the commutativity of the following diagram 
\begin{equation}
\xymatrix{
 \HH^2_{\acute {e}t}(\cM,\GG_{m,\cM}) \ar[d]_{\HH^2_{\acute {e}t}(\epsilon)} \ar[r]^{\iota^*}&  \; \HH^2_{\acute {e}t}(G,\GG_{m,G}) \ar[d]^{\HH^2_{\acute {e}t}(\epsilon_G)}\\ 
\HH^2_{\acute {e}t}(S, \GG_{m,S}) \ar@{=}[r]&  \; \HH^2_{\acute {e}t}(S, \GG_{m,S}) 
}	
\end{equation}
(since $\epsilon : S \rightarrow \cM$ is the unit section of $\cM$ and $\epsilon_G : S \rightarrow G$ is the unit section of $G$, $\iota \circ \epsilon_G = \epsilon$), $y$ is in fact an element of $\ker (\HH^2_{\acute {e}t}(\epsilon_G))(\ell).$  By Proposition \ref{mainThmG}, we know that 
 $\ker (\HH^2_{\acute {e}t}(\epsilon_G))(\ell) \subseteq \Br(G),$ and therefore the element $y$ defines a class $[A]$ in $\Br(G)$, with $A$ an Azumaya algebra on $G$.

 Via the isomorphisms $\bGerbe^2_{S(\cM)} (\GG_{m,\cM}) \cong \HH^2_{\acute {e}t}(\cM,\GG_{m,\cM})$ and $\bGerbe^2_S (\GG_{m,G}) \cong \HH^2_{\acute {e}t}(G,\GG_{m,G})$ obtained in Corollary \ref{thm:H^2}, the element $x$ corresponds to the $ \GG_{m,\cM}$-equivalence class $\overline {\cH}$ of a $ \GG_{m,\cM}$-gerbe $\cH$ on $\cM$, and 
 the element $y$ corresponds to the $ \GG_{m,G}$-equivalence class $\overline{\iota^*\cH}$ of the $ \GG_{m,G}$-gerbe $\iota^* \cH$ on $G$, which is 
 the pull-back of $\cH$ via the quotient map $\iota : G\to [G/X]$.
By the effectiveness of the 2-descent of $\GG_m$-gerbes with respect to $\iota$ proved in Lemma \ref{lem:DescentGerbes}, we can identify the $ \GG_{m,\cM}$-gerbe $\cH$ on $\cM$ with the triplet $(\iota^*\cH,\varphi,\gamma),$
where $(\varphi,\gamma)$ is 2-descent data on the $ \GG_{m,G}$-gerbe $\iota^* \cH$ 
 with respect to $\iota$. More precisely, 
$\varphi: p_2^* \iota^*\cH \rightarrow \mu^* \iota^* \cH$ is an equivalence of gerbes on~$X\times_S G$ and $\gamma: \big((\id_X\times
\mu)^*\varphi\big) \circ \big((p_{23})^* \varphi\big) \Rightarrow (m_X\times id_G)^* \varphi$ is a natural isomorphism  which satisfy the compatibility condition (\ref{eq:descentdataGerbe}). 
Moreover via the inclusions $\ker (\HH^2_{\acute {e}t}(\epsilon_G))(\ell) \hookrightarrow \Br(G) \stackrel{\delta}{\hookrightarrow} \HH^2_{\acute {e}t}(G,\GG_{m,G}),$ 
 in $\bGerbe^2_S (\GG_{m,G})$ the class $\overline{\iota^* \cH}$ coincides with the class $\overline{\delta(A)}$ of the gerbe of trivializations of the Azumaya algebra $A$.

Now we will show that the 2-descent data $( \varphi,\gamma)$ with respect to $\iota$ on the $ \GG_{m,G}$-gerbe $\iota^* \cH$ induces a descent datum 
$\varphi^A: p_2^* A \rightarrow \mu^* A$ with respect to $\iota$ on the Azumaya algebra $A$, which satisfies the cocycle condition (\ref{eq:descentdataAzumaya}).
Since the statement of our main theorem involves classes of Azumaya algebras and $ \GG_{m}$-equivalence classes  of  $ \GG_{m}$-gerbes, we may assume that
$\iota^* \cH = \delta(A)$, and so the pair $(\varphi,\gamma)$ are canonical 2-descent data on $\delta(A)$. 
The equivalence of gerbes $\varphi: p_2^* \delta(A) \rightarrow \mu^* \delta(A)$  on~$X\times_S G$ 
implies an equivalence of categories $\varphi(U): p_2^* \delta(A)(U) \rightarrow \mu^* \delta(A)(U)$  for any object  $U $ of $\bS_{\acute {e}t} $ and 
hence we have the  following diagram 
 
\begin{equation}
\xymatrix{
\uEnd_{\cO_{X \times_S G}} (p_2^*  \cL_1)  \ar[d]_{\End \varphi(U)} \ar[r]^{ \quad a_1}&  \;   p_2^* A_{|U} \\ 
\uEnd_{\cO_{X \times_S G}} (\mu^* \cL_2)   \ar[r]^{\quad a_2}&  \; \mu^* A_{|U}
}	
\end{equation}
with $\cL_1$ and $\cL_2$ objects of $\Lf(G)(U).$ 
For any object  $U $ of $\bS_{\acute {e}t} ,$ we define  $\varphi^A_{|U}:= a_2 \circ  \End \varphi(U) \circ a_1^{-1}   : p_2^* A_{|U} \rightarrow \mu^* A_{|U}$. It is an isomorphism of Azumaya algebras over $U$. The collection $(\varphi^A_{|U})_U $ of all these isomorphisms furnishes the expected isomorphism  of Azumaya algebras   $\varphi^A : p_2^* A \rightarrow \mu^* A $ on~$X\times_S G$.
For any object  $U $ of $\bS_{\acute {e}t} ,$ the natural isomorphism
$\gamma: \big((\id_X\times
\mu)^*\varphi\big) \circ \big((p_{23})^* \varphi\big) \Rightarrow (m_X\times id_G)^* \varphi$ induces
$\big((\id_X\times
\mu)^* \End \varphi(U) \big) \circ \big( (p_{23})^* \End \varphi(U) \big) = (m_X\times id_G)^*\End \varphi(U)  $ and so
  $\varphi^A$ satisfies the cocycle condition 
(\ref{eq:descentdataAzumaya}).

By Lemma \ref{lem:DescentAzumaya} the descent of Azumaya algebras with respect to $\iota$ is effective, and so the pair 
$(A, \varphi^A)$ corresponds to an Azumaya algebra $\cA$ on $\cM$, whose equivalence class $[\cA]$ is an element of $\Br(\cM)$.
\end{proof}

 \begin{corollary}\label{Br(M)}
Let $M$ be a $1$-motive defined over an algebraically closed field of characteristic zero. Then  $\Br(\cM)\cong \HH^2_{\acute {e}t}(\cM, \GG_{m,\cM} )$. 
\end{corollary}

\begin{proof} First we prove that $\HH^2_{\acute {e}t}(\cM, \GG_{m,\cM} )$ is a torsion group: in fact,
as already observed, $\HH^2_{\acute {e}t}(G, \GG_{m,G} )$ is a torsion group and by Corollary \ref{thm:H^2} $ \bGerbe^2_{S(\cM)}(\GG_{m,\cM}) \cong \HH^2_{\acute {e}t}(\cM, \GG_{m,\cM} )$  and $ \bGerbe^2_{S}(\GG_{m,G}) \cong \HH^2_{\acute {e}t}(G, \GG_{m,G} )$.
Now by Lemma \ref{lem:DescentGerbes}, the 2-descent of $\GG_m$-gerbes with respect to the quotient map
$\iota : G\to [G/X] \cong \cM$ is effective and so also the group  $\HH^2_{\acute {e}t}(\cM, \GG_{m,\cM} )$ is torsion.
Finally, if $S$ is the spectrum of an algebraically closed field of characteristic zero
$\HH^2_{\acute {e}t}(S, \GG_{m,S} )=0$, and therefore the statement is a consequence of Theorem \ref{mainTHM}.
\end{proof}

\section*{Appendix: A communication from P. Deligne on 2-descent theory for stacks}

Let $\bS$ be a site. For our applications, $\bS$ will be the site of a scheme $S$ for a Grothendieck topology, that is the category of $S$-schemes endowed with a Grothendieck topology.

Here we describe the 2-descent theory for stacks over $\bS$ following Deligne's indications. We will consider two different points of view in order to describe this 2-descent. 

We start with the point of view of covering sieves. Let $\mathcal{C}r$ be a covering sieve of the final object of $\bS$. In the case of a scheme $S$, $\mathcal{C}r$ is a sieve generated by a covering family of arrows $(X_i \to S)_i$, that is  $\mathcal{C}r$ consists of the $S$-schemes $T$ such that it exists a morphism $T \to X_i$ of $S$-schemes for an index $i$.
With these notation, the 2-descent data for a stack $\cX$ with respect to the functor $\mathcal{C}r \to \bS$ are the following:

\begin{itemize}
\item for any object $T$ in $\mathcal{C}r$, a stack $\cX_T$ over $T$ 
\item for any arrow $a : T' \to T$ between objects of $\mathcal{C}r$, an equivalence of stacks $\varphi _a : a^*\cX_T \to \cX _{T'} $ over $T'$, 
\item for any composite $T'' \stackrel{b}{\to} T' \stackrel{a} {\to} T$ in $\mathcal{C}r$, a natural isomorphism $\varphi _{ab} \Rightarrow \varphi_{b} \circ  b^* \varphi_{a} $ between equivalences of stacks over $T'',$ 
\item for any triple composite $T'''\stackrel{c}{\to}  T'' \stackrel{b}{\to} T' \stackrel{a}{\to} T$ in $\mathcal{C}r,$ a compatibility condition 
$ \varphi_{bc} \circ (bc)^* \varphi_{a} =  \varphi_{c} \circ c^* \varphi_{ab}$ involving the above natural isomorphisms.
\end{itemize}

Now we describe the 2-descent of stacks from the point of view of $\check{\mathrm{C}}$ech-coverings.
Let $S'$ be a covering of the final object of the site $\bS$ that we denote by $S$. In the case of the site of a scheme $S$ for a Grothendieck topology, $S'$ is a covering of $S$ for the chosen Grothendieck topology. The 2-descent data for a stack with respect to the covering $S' \to S$ are the following:

\begin{enumerate}
\item  a stack $\cX$ over $S'$ 
\item  over $S' \times_S S'$, an equivalence of stacks $\varphi : p_1^* \cX  \to  p_2^* \cX$, where $p_{i}: S' \times_S S' \to S'$ are the natural projections,
\item  over $S' \times_S S' \times_S S'$, a natural isomorphism $\gamma  : \ p_{23}^*\varphi \circ p_{12}^* \varphi \Rightarrow p_{13}^* \varphi$, 
 where  $p_{ij}:S' \times_S S' \times_S S' \to S' \times_S S'$ are the partial projections,
\item over $S' \times_S S' \times_S S' \times_S S'$,  the compatibility condition
\begin{equation}
p_{134}^* \gamma \circ [p_{34}^* \varphi * p_{123}^* \gamma ] =
  p_{124}^* \gamma \circ [p_{234}^* \gamma * p_{12}^* \varphi ],
\end{equation} 
where $p_{ijk}: S' \times_S S' \times_S S' \times_S S' \times_S S' \to S' \times_S S' \times_S S'$ and $p_{ij}:  S' \times_S S' \times_S S' \times_S S' \to  S' \times_S S'$ are the partial projections.
\end{enumerate}

A nice explanation of this last compatibility condition can be found in \cite[page 442 diagram (6.2.8)]{Breen90}.
 The above 2-descent of stacks through $\check{\mathrm{C}}$ech-coverings furnishes Breen's semi-local description of gerbes cited in Section 2.4.
 
 Now if $\cG$ be an $\cL$-gerbe on $\bS$, with $\cL$ an abelian band,
 the 2-descent data for $\cG$ with respect to the covering $S' \to S$ become in this case the following:

\begin{itemize}
\item[($1'$)] a neutral $\cL$-gerbe over $S'$, that is $\cTors(\cL)$, or a neutralization of  a $\cL$-gerbe $\cG'$ over $S'$, that is an equivalence of $S'$-stacks $\Psi: \, \cG'  \to \cTors(\cL) $,
\item[($2'$)] over $S' \times_S S'$, an $\cL$-torsor $T(1,2)$. In fact by \cite[Chp IV, Prop 5.2.5 (iii)]{Giraud} we have an equivalence of categories between the category of equivalences between $p_1^* \cTors(\cL)$ and $  p_2^*  \cTors(\cL)$ and the category of $\cL$-torsors.
\item[($3'$)] over $S' \times_S S' \times_S S'$, an isomorphism $\gamma: T(1,2) \wedge^\cL T(2,3) \to T(1,3)$ of $\cL$-torsors,
\item[($4'$)] over $S' \times_S S' \times_S S' \times_S S'$, the compatibility condition is that the two isomorphisms of $\cL$-torsors from 
$T(1,2) \wedge^\cL T(2,3) \wedge^\cL T(3,4)$ to $T(1,4)$ should be equal.
\end{itemize}

\par\noindent \textbf{Remark} If in ($2'$) the $\cL$-torsor $T(1,2)$ is trivial, that is $T(1,2)=\cL$, then the datum ($3'$) becomes a section of $\cL$ over $S' \times_S S' \times_S S'$ and the compatibility condition ($4'$) becomes that this section is in fact a 2-cocycle (see also end of \cite[end of 2.3]{Breen94tan}). This is the heuristic explanation of the homological interpretation of gerbes proved in Theorem \ref{thm:H^2-general}.

\bibliographystyle{plain}

\end{document}